\documentclass[11pt]{amsart}

\usepackage{amsfonts,amssymb,latexsym,amsmath,amsthm}

\newtheorem{theorem}{Theorem}[section]
\newtheorem{corollary}{Corollary}

\newtheorem{lemma}[theorem]{Lemma}
\newtheorem{proposition}{Proposition}

\theoremstyle{definition}
\newtheorem{definition}[theorem]{Definition}
\newtheorem{remark}{Remark}

\DeclareMathAlphabet{\mathpzc}{OT1}{pzc}{m}{it}

\newcommand{\Uspace}{\U }

\newcommand{\intT}{\int_0^T }
\newcommand{\gr}{>}
\newcommand{\mi}{<}

\def\half{\mbox{$\frac{1}{2}$}}

\def\supp{\mathop{\rm supp}}

\newcommand{\ddt}{\frac{\rm d}{{\rm d} t} }

\def\rh{\hat{r}}

\def\uh{\hat{u}}

\def\wh{\hat{w}}
\def\xh{\hat{x}}
\def\yh{\hat{y}}

\def\Th{\hat{T}}

\def\hb{\bar{h}}

\def\vb{\bar{v}}

\def\yb{\bar{y}}
\def\zb{\bar{z}}

\def\A{\mathcal{A}}

\def\C{\mathcal{C}}

\def\L{\mathcal{L}}

\def\P{\mathcal{P}}
\def\Q{\mathcal{Q}}
\def\R{\mathcal{R}}

\def\U{\mathcal{U}}

\def\W{\mathcal{W}}
\def\X{\mathcal{X}}
\def\Y{\mathcal{Y}}

\newcommand{\cR}{I\!\! R}
\newcommand{\cN}{I\!\!N}

\newcommand\be{\begin{equation}}
\newcommand\ee{\end{equation}}
\newcommand{\benl}{\begin{equation*}}
\newcommand{\eenl}{\end{equation*}}
\newcommand\ba{\begin{array}}
\newcommand\ea{\end{array}}
\newcommand{\bean}{\begin{eqnarray*}}
\newcommand{\eean}{\end{eqnarray*}}

\newcommand{\bs}{\bigskip}

\def\ds{\displaystyle}

\newcommand{\dtt}{\mathrm{d}t}
\newcommand{\mr}{\mathrm}

\newcommand{\lam}{[\lambda]}

\title[Quadratic order conditions for bang-singular
extremals]{Quadratic order conditions \\ for bang-singular
extremals}

\author[M.S. Aronna, J.F. Bonnans, A.V. Dmitruk and P.A. Lotito]{}


 \email{aronna@cmap.polytechnique.fr}
 \email{frederic.bonnans@inria.fr}
 \email{avdmi@cemi.rssi.ru}
 \email{plotito@exa.unicen.edu.ar}

\thanks{The first two authors are supported by the European Union under the 7th Framework Programme «FP7-PEOPLE-2010-ITN»  Grant agreement number 264735-SADCO}

\begin{document}
\maketitle

\centerline{\scshape M. Soledad Aronna }
\medskip
{\footnotesize
 \centerline{CONICET CIFASIS, Argentina}
   \centerline{INRIA Saclay - CMAP Ecole Polytechnique}
   \centerline{Route de Saclay,  91128 Palaiseau, France}
} 

\medskip

\centerline{\scshape J. Fr\'ed\'eric Bonnans}
\medskip
{\footnotesize
    \centerline{INRIA Saclay - CMAP Ecole Polytechnique}
\centerline{Route de Saclay, 91128 Palaiseau, France}
  }

\medskip

\centerline{\scshape Andrei V. Dmitruk}
\medskip
{\footnotesize
    \centerline{Russian Academy of Sciences - CEMI and Moscow State University}
 \centerline{47 Nakhimovsky Prospect, 117418 Moscow, Russia}
}

\medskip

\centerline{\scshape Pablo A. Lotito}
\medskip
{\footnotesize
    \centerline{CONICET PLADEMA - Univ. Nacional de Centro de la Prov. de Buenos Aires}
\centerline{Campus Universitario Paraje Arroyo Seco, B7000 Tandil, Argentina}
}

\bigskip


\begin{abstract}
This paper deals with optimal control problems for systems affine in the control variable. We consider nonnegativity
  constraints on the control, and finitely many equality and inequality constraints on the final state.
First, we obtain second order necessary optimality conditions. Secondly, we derive a second order sufficient condition for the scalar control case.
\end{abstract}

\section{Introduction}

In this article we obtain second order conditions for an
optimal control problem affine in the control. First we consider a pointwise nonnegativity constraint on the control,
end-point state constraints and a fixed time interval. Then we extend the result to bound constraints on the control, initial-final state constraints and problems involving parameters.
We do not assume that the  multipliers are unique. We study weak and Pontryagin minima.

There is already an important
literature on this subject. The case without control constraints, i.e. when the extremal is totally singular, has been extensively studied since the mid 1960s.
Kelley in \cite{Kel64} treated the scalar control case and presented a  necessary condition involving the second order derivative of the switching function. The result was extended by Kopp and  Moyer \cite{KopMoy65} for higher order derivatives, and in \cite{KelKopMoy67} it was shown that the order had to be even.
Goh in \cite{Goh66a} proposed a special change of variables obtained via a linear ODE and in \cite{Goh66} used this transformation to derive a necessary condition for the vector control problem. An extensive survey of these articles can be found in
Gabasov and Kirillova \cite{GabKir72}.
Jacobson and Speyer in \cite{JacSpe71}, and together with Lele in \cite{JacLelSpe71} obtained necessary conditions by adding a penalization term to the cost functional.
Gabasov and Kirillova \cite{GabKir72}, Krener \cite{Kre77}, Agrachev and Gamkrelidze \cite{AgrGam76} obtained
a countable series of necessary conditions that in fact use the idea behind the Goh transformation. Milyutin in \cite{Mil81} discovered an abstract essence of this approach and obtained
even stronger necessary conditions.
In \cite{AgrSac}  Agrachev and Sachkov investigated second order optimality conditions of the minimum time problem of a single-input system.
The main feature of this kind of problem, where the control enters linearly, is that the corresponding second variation does not contain the Legendre term, so the methods of the classical calculus of variations are not applicable for obtaining sufficient conditions. This is why the literature was mostly devoted to necessary conditions, which are actually a consequence of the nonnegativity of the second variation.
A sufficient condition for time optimality was given by  Moyer \cite{Moy73} for a system with a scalar control variable and fixed endpoints.
On the other hand,  Goh's transformation above-mentioned allows one to convert the second variation into another functional that hopefully turns out to be coercive with respect to the $L_2-$norm of some state variable.
Dmitruk in \cite{Dmi77} proved that this coercivity is a sufficient condition for the weak optimality, and presented
a closely related necessary condition. He used the abstract approach developed by Levitin, Milyutin and Osmolovskii in \cite{LMO}, and considered finitely many inequality and equality constraints on the endpoints and the possible existence of several multipliers.
In \cite{Dmi83,Dmi87} he also obtained necessary and sufficient
conditions for this norm, again closely related, for  Pontryagin minimality.
More recently, Bonnard et al. in \cite{BCT05b} provided second order sufficient conditions for the minimum time problem of a single-input system in terms of the existence of a conjugate time.

On the other hand, the case with linear control constraints and a ``purely'' bang-bang control without singular subarcs has been extensively  investigated over the past 15 years.
Milyutin and Osmolovskii in \cite{MilOsm98} provided necessary and sufficient conditions based on the general theory of \cite{LMO}.
Osmolovskii in \cite{Osm04} completed some of the proofs of the latter article.
Sarychev in \cite{Sar97} gave first and second order  sufficient condition for Pontryagin solutions.
Agrachev, Stefani, Zezza \cite{ASZ02} reduced the problem to a finite dimensional problem with the switching instants as variables and obtained a sufficient condition for strong optimality. The result was recently extended by Poggiolini and Spadini in \cite{PogSpa11}.
 On the other hand, Maurer and Osmolovskii in \cite{MauOsm03,MauOsm03b} gave a second order sufficient condition that is suitable for practical verifications and presented a numerical procedure that allows to verify the positivity of certain quadratic forms.
Felgenhauer in \cite{Fel03,Fel04,Fel05} studied both second order optimality conditions and sensitivity of the optimal solution.

The mixed case, where the control is partly bang-bang, partly singular was studied in \cite{PogSte05} by
Poggiolini and Stefani. They obtained a second order sufficient condition with an additional geometrical hypothesis (which is not  needed here) and claimed that it is not clear whether  this hypothesis is `almost necessary', in the sense that it is not obtained straightforward from a necessary condition by strengthening an inequality. In \cite{PogSte07,PogSte08} they derived a second order sufficient condition for the special case of a time-optimal problem.
The main result of the present article is to provide a sufficient condition that is `almost necessary'
for bang-singular extremals in a general Mayer problem.

On the other hand, the single-input time-optimal problem was extensively studied by means of synthesis-like methods.
See, among others, Sussmann \cite{Sus87,Sus87b,Sus87c}, Sch\"attler \cite{Sch91} and Sch\"attler-Jankovic \cite{SchJan93}. Both bang-bang and bang-singular structures were analysed in these works.

The article is organized as follows. In the second
section we present the problem and give basic
definitions. In the third section we perform a second
order analysis. More precisely, we obtain the second variation of the
Lagrangian functions and a necessary condition.
Afterwards, in the fourth section, we present the Goh transformation and a new necessary condition in the transformed variables. In the fifth section we show a sufficient condition for scalar control. Finally, we give an example with a scalar control
where the second order sufficient condition can be verified. The appendix is devoted to a series of technical properties that are used to prove the main results.

\section{Statement of the problem and assumptions}

\subsection{Statement of the problem} Consider the spaces
$\Uspace:=L_{\infty}(0,T;\mathbb{R}^m)$ and
$\mathcal{X}:=W_{\infty}^1(0,T;\mathbb{R}^n)$ as control
and state spaces, respectively. Denote with $u$ and $x$ their elements, respectively.  When needed, put $w=(x,u)$ for a point in $\mathcal{W}:=\mathcal{X}\times \Uspace.$
In this paper we investigate the optimal control problem
\begin{align}
 &\label{chap1cost} J:=\varphi_0(x(T))\rightarrow \min,\\
 &\label{chap1stateeq}\dot{x}(t)=\sum_{i=0}^mu_if_i(x),\quad
x(0)=x_0,\\
 &\label{chap1controlcons} u(t)\geq 0,\ \mr{a.e.}\ \mr{on}\
t\in [0,T],\\
 &\label{chap1finalcons}  \varphi_i(x(T))\leq 0,\ \mathrm{for}\
i=1,\hdots,d_{\varphi},\quad \eta_j(x(T))=0,\ \mathrm{for}\
j=1\hdots,d_{\eta}.
\end{align}
where $f_i:\cR^n\rightarrow
\cR^n$ for $i=0,\hdots,m,$
$\varphi_i:\cR^n\rightarrow \cR$  for $i=0,\hdots,d_{\varphi},$
$\eta_j:\cR^n\rightarrow \cR$ for $j=1,\hdots,d_{\eta}$ and $u_0\equiv 1.$ Assume that data functions $f_i$ are twice continuously
differentiable. Functions $\varphi_i$ and $\eta_j$ are
assumed to be twice differentiable.

A \textit{trajectory} is an element $w\in \mathcal{W}$
that satisfies the state equation \eqref{chap1stateeq}. If, in
addition, constraints \eqref{chap1controlcons} and \eqref{chap1finalcons}
hold, we say that $w$ is a \textit{feasible point} of the
problem \eqref{chap1cost}-\eqref{chap1finalcons}. Denote by
$\mathcal{A}$ the \textit{set of feasible points.}
A \textit{feasible variation} for $\wh\in\A$ is an element
$\delta w\in \W$ such that $\wh+\delta w\in\A.$

\begin{definition}
A pair $w^0=(x^0,u^0)\in\W$ is said to be a \textit{weak minimum} of problem \eqref{chap1cost}-\eqref{chap1finalcons} if there exists an  $\varepsilon>0$ such that the cost function attains at $w^0$ its minimum on the set
\benl
\left\{w=(x,u)\in \mathcal{A}:\|x-{x}^0\|_{\infty}<\varepsilon,\|u-u^0\|_{\infty}<\varepsilon\right\}.
\eenl
We say $w^0$ is a \textit{Pontryagin minimum} of problem \eqref{chap1cost}-\eqref{chap1finalcons} if, for any positive $N,$ there exists an $\varepsilon_N>0$ such that $w^0$ is a minimum point on the set
\benl
\left\{w=(x,u)\in \mathcal{A}:\|{x}-{x}^0\|_{\infty}<\varepsilon_N,\|u-u^0\|_{\infty}\leq N,\|u-u^0\|_1<\varepsilon_N\right\}.
\eenl
\end{definition}

Consider $\lambda=(\alpha,\beta,\psi)\in\cR^{d_{\varphi}+1,*}\times \cR^{d_{\eta},*}\times W_{\infty}^1(0,T;\cR^{n,*}),$ i.e. $\psi$ is a Lipschitz-continuous function with values in  the $n-$dimensional space of row-vectors with real components $\cR^{n,*}.$
Define the \textit{pre-Hamiltonian} function
\benl
H[\lambda](x,u,t):=\psi(t)\sum_{i=0}^mu_if_i(x),
\eenl
the \textit{terminal Lagrangian} function
\benl
\ell[\lambda](q):=\sum_{i=0}^{d_{\varphi}} \alpha_i\varphi_i(q)+\sum_{j=1}^{d_{\eta}}\beta_j \eta_j(q),
\eenl
and the \textit{Lagrangian} function
\be\label{chap1lagrangian}
\Phi[\lambda](w):= \ell[\lambda](x(T))+\int_0^{T} \psi(t)\left(\sum_{i=0}^{m}u_i(t)f_i(x(t))-\dot x(t)\right)\dtt.
\ee

In this article the optimality of a given feasible trajectory $\wh=(\xh,\uh)$ is studied. Whenever some argument of $f_i,$ $H,$ $\ell,$ $\Phi$ or their derivatives is omitted, assume that they are evaluated over this trajectory. Without loss of generality suppose that
\be
\varphi_i(\xh(T))=0,\ \mr{for}\ \mr{all}\
i=0,1,\hdots,d_{\varphi}.
\ee

\subsection{First order analysis}

\begin{definition}
Denote by $\Lambda\subset \cR^{d_{\varphi}+1,*}\times \cR^{d_{\eta},*}\times W_{\infty}^1(0,T;\cR^{n,*})$
the \textit{set of Pontryagin multipliers} associated with $\wh$ consisting of the elements $\lambda=(\alpha,\beta,\psi)$  satisfying the \textit{Pontryagin Maximum Principle,} i.e. having the following properties:
\begin{align}
\label{chap1nontriv}&|\alpha|+|\beta|=1,\\
&\label{chap1alphapos}\alpha=(\alpha_0,\alpha_1,\hdots,\alpha_{d_{\varphi}})\geq 0,
\end{align}
function $\psi$ is solution of the \textit{costate equation} and satisfies the \textit{transversality condition} at the endpoint $T,$ i.e.
\be\label{chap1costateeq}
-\dot{\psi}(t)=H_x[\lambda](\xh(t),\uh(t),t),\ \psi(T)=\ell'[\lambda](\xh(T)),
\ee
and the following \textit{minimum condition} holds
\be\label{chap1maxcond}
H[\lambda](\xh(t),\uh(t),t)=\min_{v\geq 0}H[\lambda](\xh(t),v,t), \ \mr{a.e.}\ \mr{on}\ [0,T].
\ee
\end{definition}

\begin{remark} \label{chap1Hucont}
For every $\lambda\in\Lambda,$ the following two conditions hold.
\begin{itemize}
\item[(i)] $H_{u_i}\lam(t)$ is continuous in time,
\item[(ii)] $H_{u_i}[\lambda](t)\geq 0,$ a.e. {on} $[0,T].$
\end{itemize}
\end{remark}

Recall the following well known result for which a proof can be found e.g. in  Alekseev and Tikhomirov \cite{AleTikFom79}, Kurcyusz and Zowe \cite{KurZow}.
\begin{theorem}
 The set $\Lambda$ is not empty.
\end{theorem}

\begin{remark}\label{chap1Lambdacompact}
Since $\psi$ may be expressed as a linear continuous mapping of $(\alpha,\beta)$ and since \eqref{chap1nontriv} holds, $\Lambda$ is a finite-dimensional compact set. Thus, it can be identified with a compact subset of $\cR^s,$ where $s:=d_{\varphi}+d_{\eta}+1.$
\end{remark}

The following expression for the derivative of the Lagrangian function holds
\be
\label{chap1derlag}
\Phi_u[\lambda](\wh)v=\int_0^TH_u[\lambda](\xh(t),\uh(t),t)v(t)\dtt.
\ee
Consider $v\in \U$ and the \textit{linearized state equation:} 
\begin{equation} \label{chap1lineareq}
\left\{
\begin{array}{l}
\dot{z}(t)=\ds\sum_{i=0}^m \uh_i(t)f_i'(\xh(t))z(t)+\sum_{i=1}^mv_i(t)f_i(\uh(t)),\quad \mr{a.e.}\,\mr{on}\,[0,T],\\
z(0)=0.
\end{array}
\right.
\end{equation}
Its solution $z$ is called the \textit{linearized state variable.}

With each index $i=1,\hdots,m,$ we associate the sets
\be
\label{chap1setsI}
I_0^i:=\left\{t\in [0,T]: \max_{\lambda\in\Lambda}H_{u_i}\lam(t)>0\right\},\quad
I_+^i:=[0,T]\backslash I_0^i,
\ee
and the \textit{active set}
\be
\tilde I^i_0:=\{t\in [0,T]:\uh_i(t)=0\}.
\ee
Notice that $I_0^i\subset \tilde I^i_0,$ and that $I_0^i$ is relatively open in $[0,T]$ as each $H_{u_i}\lam$ is continuous.

\noindent \textbf{Assumption 1.}
Assume
\textit{strict complementarity for the control constraint,}
i.e.  for every $i=1,\hdots,m,$
\be \label{chap1sc}
I^i_0=\tilde I^i_0,\
\mr{up}\ \mr{to}\ \mr{a}\ \mr{set}\ \mr{of}\ \mr{null}\ \mr{measure.}
\ee

Observe then that for any index $i=1,\hdots,m,$ the control $\uh_i(t)\gr 0$ a.e. on $I_+^i,$ and given $\lambda\in \Lambda,$
\benl
H_{u_i}\lam(t)=0,\ \mr{a.e.}\ \mr{on}\ I_+^i.
\eenl

\noindent \textbf{Assumption 2.}
For every
$i=1,\hdots,m,$ the active set $I^i_0$ is a finite union of
intervals, i.e.
\benl
I^i_0=\ds\bigcup_{j=1}^{N_i} I^i_{j},
\eenl
for $I^i_{j}$ subintervals of $[0,T]$ of the form $[0,d),$
$(c,T];$ or $(c,d)$ if $c\neq 0$ and $d\neq T.$
Denote by
$c_1^i<d_1^i<c_2^i<\hdots<c_{N_i}^i<d_{N_i}^i$ the
endpoints of these intervals. Consequently,
$I_+^i$ is a finite union of intervals as well.

\begin{remark}[On the multi-dimensional control case]

We would like to make a comment concerning solutions with more than one control component being singular at the same time. In \cite{CJT06,CJT08}, Chitour et al. proved that generic systems with three or more control variables, or with two controls and drift did not admit singular optimal trajectories (by means of Goh's necessary condition \cite{Goh66}).  Consequently, the study of generic properties of control-affine systems is restricted to problems having either one dimensional control or two control variables and no drift.
Nevertheless, there are motivations for investigating problems with an arbitrary number of inputs that we point out next.
In \cite{LedSch12}, Ledzewicz and Sch\"attler worked on a model of cancer treatment having two control variables entering linearly in the pre-Hamiltonian and nonzero drift. They provided necessary optimality conditions for solutions with both controls being singular at the same time. Even if they were not able to give a proof of optimality they claimed to have strong expectations that this structure is part of the solution.
Other examples can be found in the literature. Maurer in \cite{Mau76} analyzed a resource allocation problem (taken from Bryson-Ho \cite{BryHo}). The model had two controls and drift, and numerical computations yielded a candidate solution containing two simultaneous singular arcs.
For a system with a similar structure,
Gajardo et al. in \cite{GRR08} discussed the optimality
of an extremal with two singular control components at the same time.
Another motivation that we would like to point out is the technique used in Aronna et al. \cite{ABM11} to study the shooting algorithm for bang-singular solutions. In order to treat this kind of extremals, they perform a transformation that yields a new system and an associated totally singular solution.
This new system involves as many control variables as singular arcs of the original solution. Hence, even a one-dimensional problem can lead to a multi-dimensional totally singular solution. These facts give a motivation for the investigation of multi-input control-affine problems.
\end{remark}


\subsection{Critical cones}
Let $1\leq p \leq \infty,$ and call $\Uspace_p:=L_p(0,T;\cR^m),$   $\Uspace_p^+:=L_p(0,T;\cR^m_+)$ and $\X_p:=W^1_p(0,T;\cR^n).$
Recall that given a topological vector space $E,$ a subset $D\subset E$ and $x\in E,$ a \textit{tangent direction} to $D$ at $x$ is an element $d\in E$ such that there exists sequences $(\sigma_k)\subset \cR_+$ and $(x_k)\subset D$ with
\benl
\frac{x_k-x}{\sigma_k}\rightarrow d.
\eenl
It is a well known result, see e.g. \cite{ComPen97}, that the tangent cone to $\U_2^+$ at $\uh$ is
\benl
\label{chap1tangentL2}
\{v\in\U_2:\ v_i\geq 0\ \mathrm{on}\ I^i_0,\ \mr{for}\, i=1,\hdots,m\}.
\eenl
Given $v\in\U_p$ and $z$  the solution of \eqref{chap1lineareq}, consider the \textit{linearization of the cost and final constraints}
\be\label{chap1linearcons}
\left\{
\begin{split}
&\varphi'_i(\xh(T))z(T)\leq 0,\ i=0,\hdots,d_{\varphi},\\
&\eta'_j(\xh(T))z(T)=0,\ j=1,\hdots,d_{\eta}.
\end{split}
\right.
\ee
For $p\in\{2,\infty\},$ define the $L_p-$\textit{critical cone} as
\benl
\label{chap1critcones}
\C_p:=
\left\{
(z,v)\in \X_p\times\Uspace_p:
v\ \mr{tangent}\ \mr{to}\  \Uspace_p^+,\ \text{\eqref{chap1lineareq} and \eqref{chap1linearcons}}\ \mr{hold}
\right\}.
\eenl
Certain relations of inclusion and density between
some approximate critical cones are needed. Given $\varepsilon\geq0$ and $i=1,\hdots,m,$ define the \textit{$\varepsilon-$active sets,} up to a set of null measure
\benl
I_{\varepsilon}^i:=\{t\in (0,T):\uh_i(t)\leq
\varepsilon\},
\eenl
and the sets
\benl
\W_{p,\varepsilon}:=\{(z,v)\in \X_p\times\Uspace_p:v_i=0\ \mathrm{on}\ I^i_{\varepsilon},\ \eqref{chap1lineareq}\ \mr{holds}\}.
\eenl
By Assumption 1, the following explicit expression for $\C_2$ holds
\be
\label{chap1C2}
\C_2=
\{
(z,v)\in\W_{2,0}:\eqref{chap1linearcons}\ \mr{holds}\}.
\ee
Consider the \textit{$\varepsilon-$critical cones}
\be
\label{chap1epsiloncone}
\C_{p,\varepsilon}:=
\{(z,v)\in \W_{p,\varepsilon}:\eqref{chap1linearcons}\ \mr{holds}\}.
\ee
Let $\varepsilon\gr 0.$ Note that by \eqref{chap1C2}, $\C_{2,\varepsilon}\subset \C_2.$ On the other hand, given $(z,v)\in \C_{\infty,\varepsilon},$ it easily follows that $\uh+\sigma v\in\U^+$ for small positive $\sigma.$ Thus $v$ is tangent to $\U^+$ at $\uh,$ and this yields $\C_{\infty,\varepsilon}\subset\C_{\infty}.$

\if{
\begin{remark}
 \label{chap1phi0}  Let $(z,v)\in \C_2.$ Observe that since $v_i=0$ on $I_0^i,$ necessarily $\intT H_u\lam(t)v(t)\dtt=0$ and thus, by \eqref{chap1derlag}, $\Phi_u\lam(\wh)v=0.$ Consequently,
\be
\sum_{i=0}^{d_{\varphi}} \alpha_i\varphi'(\xh(T))z(T)
+
\sum_{j=1}^{d_{\eta}} \beta_j \eta_j'(\xh(T))z(T)=0.
\ee
In case $\alpha_0>0,$ note that if \eqref{chap1linearcons} holds only for $i\geq 1,$ then $\varphi_0'(\xh(T))z(T)\leq 0.$
\end{remark}
}\fi

\bs

Recall the following technical result, see Dmitruk \cite{Dmi08}.
\begin{lemma}[on density]
\label{chap1lemmadense}
Consider a locally convex topological space $X,$ a
finite-faced cone $C\subset X,$ and a linear manifold $L$
dense in $X.$ Then the cone $C\cap L$ is dense in $C.$
\end{lemma}

\if{
Let the cone $C$ be given by

\be
C=\{x\in X:(p_i,x)=0,\ \mathrm{for}\ i= 1,\hdots,\mu,\
(q_j,x)\leq 0,\ \mathrm{for}\ j=1,\hdots,\nu\}.
\ee

 Let us show first, without lost of generality, that the
equality constraints can be removed from the formulation.
It suffices to consider the case where $C$ is given by only
one equality $(p,x)=0.$

 Take any point $x_0\in C$ and a convex neighborhood
$\mathcal{O}(x_0).$ We have to show that there exists $x$
in $C\cap L\cap \mathcal{O}(x_0).$ Since the set $(p,x)<0$
is open, its intersection with $\mathcal{O}(x_0)$ is open
too and obviously nonempty, hence it contains a point $x_1$
from the set $L,$ because the last one is dense in $X.$
Similarly, the intersection of the sets $(p,x)<0$ and
$\mathcal{O}(x_0)$ contains a point $x_2\in L.$ Since
$\mathcal{O}(x_0)$ is convex, it contains a point $x$ such
that $(p,x)=0,$ which belongs to $C$ and to $L\cap
\mathcal{O}(x_0).$

 We can then consider only the case where $C$ is given by a
finite number of inequalities. Suppose first that there
exists $\xh\in C$ such that $(q_j,\xh)<0$ for all $j,$
hence $\xh \in \mathrm{int}{C}.$ Take any $x_0\in C$ and
any convex neighborhood $\mathcal{O}(x_0).$ We have to find
a point $x\in C\cap \mathcal{O}(x_0)\cap L.$ We know that,
for any positive $\varepsilon,$ the point
$x_{\varepsilon}:=x_0+\varepsilon \xh$ lies in
$\mathrm{int}(C),$ and then there exists a positive
$\varepsilon$ such that this point lies also in
$\mathcal{O}(x_0).$ Thus, the open set $\mathrm{int}{C}\cap
\mathcal{O}(x_0)$ is nonempty, and then contains a point
$x$ from the dense set $L.$

 Suppose now that the above point $\xh\in\mathrm{int}{C}$
does not exist, and, without lost of generality,  that
$q_j\neq 0$ for every $j.$ In this case, by the
Dubovitskii-Milyutin theorem, there exist  multipliers
$\alpha_j\geq0,\ j=1,\hdots,\nu,$ not all zero, such that
Euler-Lagrange equation holds: $\alpha_1
q_1+\hdots+\alpha_{\nu}q_{\nu}=0.$ Suppose, without lost of
generality, that $\alpha_{\nu}>0.$ Then, for all $x \in C$
we actually have $(q_{\nu},x)=0,$ not just $\leq0.$  This
means that the cone $C$ can be given by the constraints
$(q_j,x)\leq 0,\ j=1,\hdots,\nu-1,\ (q_{\nu},x)=0.$ But, as
was already shown, the last equality can be removed, so the
cone can be given by a smaller number of inequalities.
Applying induction arguments, we arrive at a situation when
either all the inequalities are changed into equalities and
then removed, or the strict inequalities have a nonempty
intersection. Since both cases are already considered, the
proof is complete.
}\fi

\begin{lemma}\label{chap1conedense}
 Given $\varepsilon>0$ the following properties hold.
\begin{itemize}
 \item [(a)] $\C_{\infty,\varepsilon}\subset
\C_{2,\varepsilon}$ with dense inclusion.
\item [(b)]
$\bigcup_{\varepsilon>0}\C_{2,\varepsilon}\subset \C_2$
with dense inclusion.
\end{itemize}
\end{lemma}

\begin{proof} \textbf{(a)} The inclusion is immediate.
As $\Uspace$ is dense in $\Uspace_2,$ $\W_{\infty,\varepsilon}$ is a dense subspace of $\W_{2,\varepsilon}.$ By Lemma \ref{chap1lemmadense}, $\C_{2,\varepsilon}\cap \W_{\infty,\varepsilon}$ is dense
in $\C_{2,\varepsilon},$  as desired.

\noindent\textbf{(b)} The inclusion is immediate. In order to
prove density, consider the following dense subspace of $\W_{2,0}:$
\benl
\W_{2,\bigcup}:=\bigcup_{\varepsilon>0} \W_{2,\varepsilon},
\eenl
and the finite-faced cone in $\C_2\subset \W_{2,0}.$
By Lemma \ref{chap1lemmadense}, $\C_2\cap \W_{2,\bigcup}$ is dense in $\C_2,$
which is what we needed to prove.
\end{proof}


\section{Second order analysis}

\subsection{Second variation}

Consider the following quadratic mapping on $\W;$
\benl
\label{chap1Omega}
\begin{split}
\Omega[\lambda](\delta x,\delta
u):=&\half\ell''[\lambda](\xh(T))(\delta x(T))^2\\
&+
\half\int_0^T \left[(H_{xx}[\lambda]\delta x,\delta
x)+2(H_{ux}[\lambda]\delta x,\delta u)\right]\dtt.
\end{split}
\eenl
The next lemma provides a second order expansion for the Lagrangian function involving operator $\Omega.$
Recall the following notation: given two functions $h:\cR^n\rightarrow\cR^{n_h} $ and
$k:\cR^n\rightarrow \cR^{n_k},$ we say
that $h$ is a \textit{big-O} of $k$ around 0 and denote
it by
\benl
h(x)=\mathcal{O}(k(x)),
\eenl
if there exists positive constants $\delta$ and $M$ such that
$|h(x)|\leq M|k(x)|$ {for}  $|x|<\delta.$
It is a \textit{small-o} if $M$ goes to 0 as $|x|$ goes
to 0. Denote this by
\benl
h(x)=o(k(x)).
\eenl

\begin{lemma}
\label{chap1expansionlagrangian}
Let $\delta w=(\delta x,\delta
u)\in\W.$ Then for every multiplier
$\lambda\in\Lambda,$ the function $\Phi$ has the following expansion (omitting time arguments):
\begin{align*}
\Phi[\lambda](\wh+\delta w)=& \int_0^T H_u[\lambda]\delta u\dtt+\Omega[\lambda](\delta
x,\delta u)+\half\int_0^T  (H_{uxx}[\lambda]\delta x,\delta x,\delta u)\dtt\\
&+\mathcal{O}(|\delta x(T)|^{3})+
\int_0^T |(\uh+\delta u)(t)| \mathcal{O}(|\delta x(t)|^3)\,  \dtt.
\end{align*}
\end{lemma}

\begin{proof}
Omit the dependence on $\lambda$ for the sake of simplicity.
Use the Taylor expansions
\benl\label{chap1lexp}
 \ell(\xh(T)+\delta x(T))
=\ell(\xh(T))+\ell'(\xh(T))\delta
x(T)+\half\ell''(\xh(T))(\delta x(T))^2+\mathcal{O}(|\delta x(T)|^{3}),
\eenl
\benl\label{chap1fexp}
 f_i(\xh(t)+\delta x(t))
=f_i(\xh(t))+f_i'(\xh(t))\delta x(t)+\half
f_i''(\xh(t))(\delta x(t))^2+\mathcal{O}(|\delta x(t)|^3),
\eenl
 in the expression
\benl\label{chap1DeltaPhi}
\Phi(\wh+\delta w)=\ell(\xh+\delta x(T))+\intT \psi\left[
\sum_{i=0}^m(\uh_i+\delta u_i)f_i(\xh+\delta x)-\dot{\xh}-\dot{\delta x}\right]\dtt.
\eenl
Afterwards, use the identity
\benl
\int_0^T \psi\sum_{i=0}^m \uh_if'_i(\xh)\delta x \dtt
=
-\ell'(\xh(T))\delta x(T)+\int_0^T \psi\dot{\delta x}\dtt,
\eenl
obtained by integration by parts and equation \eqref{chap1stateeq} to get the desired result.
\end{proof}

The previous lemma yields the following identity for  every $(\delta x,\delta u)\in \W:$
 \benl
\Omega[\lambda](\delta x,\delta
u)=\half D^2\Phi[\lambda](\wh)(\delta x,\delta u)^2.
\eenl

\subsection{Necessary condition} This section provides the following second order necessary condition in terms of $\Omega$ and the critical cone $\C_2.$
\begin{theorem}
\label{chap1NC2}
If $\wh$ is a weak minimum then
\be
\label{chap1NC2eq}
\max_{\lambda\in\Lambda}\Omega[\lambda](z,v)\geq 0,\quad
\mr{for}\ \mr{all}\ (z,v)\in \C_2.
\ee
\end{theorem}
For the sake of simplicity, define  $\bar\varphi:\Uspace\rightarrow \cR^{d_{\varphi}+1},$ and
$\bar\eta:\Uspace\rightarrow \cR^{d_{\eta}}$ as
\be
\label{chap1barvarphi}
\begin{split}
\bar\varphi_i(u)&:=\varphi_i(x(T)),\ \mr{for}\
i=0,1,\hdots,d_{\varphi},\\
\bar\eta_{j}(u)&:=\eta_j(x(T)),\ \mr{for}\
j=1,\hdots,d_{\eta},
\end{split}
\ee
where $x$ is the solution of \eqref{chap1stateeq} corresponding
to $u.$

\begin{definition}\label{chap1qcdef}
We say that the \textit{equality constraints are nondegenerate} if
\be
\label{chap1cq}
\bar\eta'(\uh)\ \mathrm{is}\ \mathrm{onto}\ \mr{from}\ \U\ \mr{to}\,\cR^{d_{\eta}}.
\ee
If \eqref{chap1cq} does not hold, we call them \textit{degenerate}.
\end{definition}



Write the problem in the following way
\be\label{chap1P}\tag{P}
\bar\varphi_0(u)\rightarrow \min;\quad \bar\varphi_i(u)\leq 0,\ i=1,\hdots,d_{\varphi},\
\bar\eta(u)=0,\ u\in \Uspace_+.
\ee
\textbf{Suppose that $\uh$ is a local weak solution of \eqref{chap1P}.}  Next we prove Theorem \ref{chap1NC2}. Its proof is divided into two cases: degenerate and nondegenerate equality constraints. For the first case the result is immediate and is tackled in the next Lemma. In order to show Theorem \ref{chap1NC2} for the latter case we introduce an auxiliary problem parameterized by certain critical directions $(z,v),$
denoted by \eqref{chap1QPv}.
We prove that val$\eqref{chap1QPv}\geq0$ and, by a result on duality, the desired second order condition will be derived.

\begin{lemma}\label{chap1notqualified}
If equality constraints are degenerate, then \eqref{chap1NC2eq} holds.
\end{lemma}

\begin{proof}
Notice that there exists $\beta\neq 0$ such that $\sum_{j=1}^{d_{\eta}}\beta_j\eta_j'(\xh(T))=0,$ since  $\bar{\eta}'(\uh)$ is not onto.  Consider $\alpha=0$ and $\psi=0.$ Take $\lambda:=(\alpha,\beta,\psi)$ and notice that both $\lambda$ and $-\lambda$ are in $\Lambda.$
Observe that
\benl
\Omega\lam(z,v)=\half\sum_{j=1}^{d_{\eta}}\beta_j\eta_j''(\xh(T))(z(T))^2.
\eenl
Thus $\Omega\lam(z,v)\geq 0$ either for $\lambda$ or $-\lambda.$ The required result follows.
\end{proof}

\if{

Consider the following auxiliary problem, where variable $\zeta$ is in $\cR:$
\be\label{chap1Pzeta}\tag{P$_{\zeta}$}
\begin{split}
\zeta\rightarrow \min,\quad&\bar\varphi_0(u)-\bar\varphi_0(\uh)\leq \zeta,\ \bar\varphi_i(u)\leq \zeta,\ i=1,\hdots,d_{\varphi},\\ &\bar\eta(u)=0,\ -u(t)\leq \zeta,\ \mr{a.e}.
\end{split}
\ee

\begin{proposition} If $\uh$ is a local weak solution of \eqref{chap1P} then $(\uh,0)$ is a local solution of \eqref{chap1Pzeta} in the $\U\times \cR-$ norm.
\end{proposition}

\begin{proof}
Let us show that there exists a neighborhood of $(\uh,0)$ where $(\uh,0)$ is a minimum.
As $\uh$ is a local weak solution of \eqref{chap1P}, there exists a neighbourhood $\mathcal{O}$ of $\uh$ where it is a minimum. Take $u\in\mathcal{O}$ and $\zeta\in\cR,$ and show that $(\uh,0)$ is a minimum in this set. If $\zeta\geq0$ then $(u,\zeta)$ is not a better solution than $(\uh,0),$ thus we have to consider $\zeta<0.$ Hence
\be
\bar\varphi_0(u)\leq \zeta+\bar\varphi_0(\uh)<\bar\varphi_0(\uh),
\ee
and $u$ is feasible for \eqref{chap1P}. This leads us to a contradiction.
\end{proof}

}\fi

Take $\varepsilon>0,$ $(z,v)\in \C_{\infty,\varepsilon},$ and rewrite \eqref{chap1epsiloncone} using the notation in \eqref{chap1barvarphi},
\benl
\begin{split}
\C_{\infty,\varepsilon}
=\{&(z,v)\in \X\times \Uspace:\
v_i(t)=0\ \mr{on}\ I_{\varepsilon}^i,\ \ i=1,\hdots,m,\\
&\eqref{chap1lineareq}\ \mr{holds},\
\bar\varphi'_i(\uh)v\leq 0,\ i=0,\hdots,d_{\varphi},\ \bar\eta'(\uh)v=0\}.
\end{split}
\eenl
Consider the problem
\be\label{chap1QPv}\tag{QP$_v$}
\begin{split}
&\delta \zeta\rightarrow \min\\
&\bar\varphi_i'(\uh)r+\bar\varphi_i''(\uh)(v,v)\leq \delta \zeta,\,\mr{for}\ i=0,\hdots,d_{\varphi},\\
&\bar\eta'(\uh)r+\bar\eta''(\uh)(v,v)=0,\\
&-r_i(t)\leq \delta \zeta,\ \mr{on}\ I^i_0,\ \mr{for}\ i=1,\hdots,m.
\end{split}
\ee

\begin{proposition} \label{chap1valgeq0} Let $(z,v)\in \C_{\infty,\varepsilon}.$
If the equality constraints are nondegenerate,
problem \eqref{chap1QPv} is feasible and ${\rm val}\,\eqref{chap1QPv}\geq 0.$
\end{proposition}

\begin{proof}
Let us first prove feasibility.
As $\bar\eta'(\uh)$ is onto, there exists $r\in\U$ such that the
equality constraint in \eqref{chap1QPv} is satisfied.
Take
\benl
\delta\zeta:=\max(\|r\|_{\infty},\bar\varphi_i'(\uh)r+\bar\varphi''(\uh)(v,v)).
\eenl
Thus the  pair $(r,\delta \zeta)$ is feasible for \eqref{chap1QPv}.

Let us now prove that ${\rm val}\,\eqref{chap1QPv}\geq 0.$ On the contrary suppose that there exists a feasible solution $(r,\delta \zeta)$ with $\delta \zeta<0.$ The last constraint in \eqref{chap1QPv} implies $\|r\|_{\infty}\neq0.$
 Set, for $\sigma>0,$
\be
\label{chap1usigma}
\tilde u(\sigma):=\uh+\sigma v+\half\sigma^2 r,\quad \tilde\zeta(\sigma):=\half \sigma^2\delta\zeta.
\ee
The goal is finding $u(\sigma)$ feasible for \eqref{chap1P} such that for small $\sigma,$
\benl
u(\sigma)\overset{\U}\rightarrow \uh,\,\,\mr{and}
\,\, \bar\varphi_0(u(\sigma))< \bar\varphi_0(\uh),
\eenl
contradicting the weak optimality of $\uh.$

Notice that $\uh_i(t)> \varepsilon$ a.e. on $[0,T]\backslash I_{\varepsilon}^i,$ and then
$\tilde u(\sigma)_i(t)>-\tilde \zeta(\sigma)$ for sufficiently small $\sigma.$
On $I_{\varepsilon}^i,$ if $\tilde u(\sigma)_i(t)<-\tilde \zeta(\sigma)$
then necessarily
\benl
 \uh_i(t)< \half\sigma^2(\|r\|_{\infty}+|\delta\zeta|),
\eenl
 as $v_i(t)=0.$
Thus, defining the set
\benl
J_{\sigma}^i:=\{t:0<\uh_i(t)<\half\sigma^2(\|r\|_{\infty}+|\delta\zeta|)\},
\eenl
we get $\{t\in[0,T]:\tilde u(\sigma)_i(t)<-\tilde \zeta(\sigma)\}
\subset J_{\sigma}^i.$
Observe that on $J_{\sigma}^i,$ the function $|\tilde u(\sigma)_i(t)+\tilde \zeta(\sigma)|/\sigma^2$ is dominated by $\|r\|_{\infty}+|\delta\zeta|.$ Since $\mr{meas}(J_{\sigma}^i)$ goes to 0 by the Dominated Convergence Theorem, we obtain
\benl
\label{chap1u1bound0}
\int_{J_{\sigma}^i}  |\tilde u(\sigma)_i(t)+\tilde \zeta(\sigma) |\dtt= o(\sigma^2).
\eenl
Take
\benl
\tilde{\tilde u}(\sigma):=
\left\{
\ba{rl}
\tilde{u}(\sigma)&\mr{on}\ [0,T]\backslash J_{\sigma}^i,\\
-\tilde\zeta(\sigma)& \mr{on}\ J_{\sigma}^i.
\ea
\right.
\eenl
Thus, $\tilde{\tilde u}$ satisfies
\be
\label{chap1strictpos}
\tilde{\tilde u}(\sigma)(t)\geq -\tilde \zeta(\sigma),\quad \mr{a.e.}\ \mr{on}\ [0,T],
\ee
\benl
\|\tilde{\tilde u}(\sigma)-\uh\|_1=o(\sigma^2),\quad \|\tilde{\tilde u}(\sigma)-\uh\|_{\infty}=O(\sigma^2),
\eenl
and the following estimates hold
\be
\label{chap1varphiest}
\begin{split}
\bar\varphi_i(\tilde {\tilde u}(\sigma))
&=
\bar\varphi_i(\uh)+\sigma\bar\varphi_i'(\uh)v
+\half\sigma^2[\bar\varphi_i'(\uh)r+\bar\varphi''_i(\uh)(v,v)]+o(\sigma^2)\\
&< \bar\varphi_i(\uh)+\tilde\zeta(\sigma)+o(\sigma^2),
\end{split}
\ee
\benl
\label{chap1etaest}
\bar\eta(\tilde {\tilde u}(\sigma))=
\sigma\bar\eta'(\uh)v
+\half\sigma^2[\bar\eta'(\uh)r+\bar\eta''(\uh)(v,v)]+o(\sigma^2)=o(\sigma^2).
\eenl
\if{
Notice that for $t\in I_0^i:$
\be
-\tilde u(\sigma)_i(t)=-\half\sigma^2r_i(t)< \tilde\zeta(\sigma),
\ee
and for $t\in (0,T)\backslash I_{\varepsilon}^i:$
\be
-\tilde u(\sigma)_i(t)=-\uh_i(t)-\sigma v_i(t)-\half\sigma^2r_i(t)<-\half\varepsilon<\tilde \zeta(\sigma),\quad \mr{for}\ \mr{small}\ \sigma,
\ee
as $\uh_i(t)>\varepsilon.$  Thus $
\{t\in[0,T]:-\tilde u(\sigma)_i(t)>\tilde \zeta(\sigma)\}
\subset
I_{\varepsilon}^i\backslash I_0^i.
$
}\fi
As $\bar\eta'(\uh)$ is onto on $\U$ we can find a corrected control $u(\sigma)$ satisfying the equality constraint and such that $\|u(\sigma)-\tilde{\tilde u}(\sigma)\|_{\infty}=o(\sigma^2).$  Deduce by \eqref{chap1strictpos} that $u(\sigma)\geq 0$ a.e. on $[0,T],$ and by \eqref{chap1varphiest} that it satisfies the terminal inequality constraints. Thus $u(\sigma)$ is feasible for \eqref{chap1P} and it satisfies \eqref{chap1usigma}. This contradicts the weak optimality of $\uh.$
\end{proof}

Recall that a \textit{Lagrange multiplier} associated with $\wh$ is a pair $(\lambda,\mu)$ in $\cR^{d_{\varphi}+1}\times \cR^{d_{\eta}}\times W_{\infty}^1(0,T;\cR^{n,*})\times \U^*$ with
$\lambda=(\alpha,\beta,\psi)$ satisfying \eqref{chap1nontriv}, \eqref{chap1alphapos}, $\mu\geq 0$ and the \textit{stationarity condition}
\benl
\label{chap1HSC}
\intT H_u\lam(t)v(t)\dtt+\intT v(t)\mr{d}\mu(t)=0,\quad \mr{for}\ \mr{every}\ v\in \U.
\eenl
Here $\U^*$ denotes the dual space of $\U.$
Simple computations show that $(\lambda,\mu)$ is a Lagrange multiplier if and only if $\lambda$ is a Pontryagin multiplier and $\mu=H_u\lam.$ Thus $\mu\in L_{\infty}(0,T;\cR^{m,*}).$

Let us come back to Theorem \ref{chap1NC2}.

\begin{proof}

[of Theorem \ref{chap1NC2}]
Lemma \ref{chap1notqualified} covers the degenerate case.
Assume thus that $\bar\eta'(\uh)$ is onto. Take $\varepsilon\gr 0$  and $(z,v)\in \C_{\infty,\varepsilon}.$   Applying Proposition \ref{chap1valgeq0}, we see that there cannot exist $r$ and $\delta \zeta<0$ such that
\benl
\begin{split}
&\bar\varphi_i'(\uh)r+\bar\varphi_i''(\uh)(v,v)\leq \delta \zeta,\ i=0,\hdots,d_{\varphi},\\
&\bar\eta'(\uh)r+\bar\eta''(\uh)(v,v)=0,\\
&-r_i(t)\leq \delta \zeta,\ \mr{on}\ I^i_0,\ \mr{for}\ i=1,\hdots,m.
\end{split}
\eenl
By the Dubovitskii-Milyutin Theorem (see \cite{DubMil}) we obtain the existence of $(\alpha,\beta)\in \cR^s$ and $\mu\in \U^*$ with $\mr{supp}\ \mu_i\subset I_0^i,$ and $(\alpha,\beta,\mu)\neq 0$ such that
\be\label{chap1firstorder}
\sum_{i=0}^{d_{\varphi}}\alpha_i\bar\varphi_i'(\uh)+\sum_{i=1}^{d_{\eta}}
\beta_j\bar\eta_j'(\uh)
-\mu=0,
\ee
and denoting $\lambda:=(\alpha,\beta,\psi),$ with $\psi$ being solution of \eqref{chap1costateeq}, the following holds:
\benl
\sum_{i=0}^{d_{\varphi}}\alpha_i\bar\varphi_i''(\uh)(v,v)
+\sum_{i=1}^{d_{\eta}}
\beta_j\bar\eta_j''(\uh)(v,v)\geq 0.
\eenl
By Lemma \ref{chap1barell} we obtain
\be\label{chap1SecondDer}
\Omega\lam(z,v)\geq 0.
\ee
Observe that \eqref{chap1firstorder} implies that $\lambda\in\Lambda.$
Consider now $(\zb,\vb)\in \C_2,$ and note that Lemma \ref{chap1conedense} guarantees the existence of a sequence
$\{(z_{\varepsilon},v_{\varepsilon})\}\subset \C_{\infty,\varepsilon}$ converging to $(\zb,\vb)$ in $\X_2\times \U_2.$
Recall Remark \ref{chap1Lambdacompact}. Let $\lambda_{\varepsilon}\in\Lambda$ be such that  \eqref{chap1SecondDer} holds for $(\lambda_{\varepsilon},z_{\varepsilon},v_{\varepsilon}).$
Since $(\lambda_{\varepsilon})$ is bounded, it contains a
limit point $\bar{\lambda}\in\Lambda.$ Thus \eqref{chap1SecondDer} holds for $(\bar\lambda,\zb,\vb),$ as required.
\end{proof}

\section{Goh Transformation}

Consider an arbitrary linear system: 
\begin{equation} \label{chap1lineareqG}
\left\{
\begin{array}{l}
\dot{z}(t)=A(t)z(t)+B(t)v(t),\quad {\rm a.e.}\ {\rm on}\ [0,T],\\
z(0)=0,
\end{array}
\right.
\end{equation}
where $A(t)\in \L(\cR^n;\cR^n)$ is an essentially bounded function of $t,$ and $B(t)\in
\L(\cR^m;\cR^n)$ is a Lipschitz-continuous function of $t.$
With each $v\in \Uspace$ associate the
state variable $z\in \mathcal{X}$ solution of \eqref{chap1lineareq}. Let us present a transformation of the variables $(z,v)\in\W,$ first introduced by Goh in \cite{Goh66a}.
Define two new state variables as follows:
\be \label{chap1yxi}
\left\{
\begin{split}
y(t)&:=\ds\int_0^t v(s)\mathrm{ds},\\
{\xi}(t)&:=z(t)-B(t)y(t).
\end{split}
\right.
\ee
Thus $y\in \mathcal{Y}:=W_{\infty}^1(0,T;\cR^m),$ $y(0)=0$ and $\xi$ is an element of space $\X.$ It easily follows that $\xi$ is  a solution of the linear differential equation
\be\label{chap1xieq}
\dot{\xi}(t)=A(t)\xi(t)+B_1(t)y(t),\quad  \xi(0)=0,
\ee
where
\be
\label{chap1B1}
B_1(t):=A(t)B(t)-\dot{B}(t).
\ee
For the purposes of this article take
\be\label{chap1AB}
A(t):=\sum_{i=0}^m \uh_if_i'(\xh(t)),\ \mr{and}\
B(t)v(t):=\sum_{i=1}^mv_i(t)f_i(\uh(t)).
\ee
Then \eqref{chap1lineareqG} coincides with the linearized equation \eqref{chap1lineareq}.

\subsection{Transformed critical directions}
As optimality conditions on the variables
obtained by the Goh Transformation will be derived,
a new set of critical directions is needed. Take a point
$(z,v)$ in $\C_{\infty},$ and define $\xi$ and $y$ by the  transformation \eqref{chap1yxi}. Let $h:=y(T)$ and notice that since \eqref{chap1linearcons} is satisfied, the following inequalities hold,
\be
\label{chap1translinearcons}
\begin{split}
&\varphi_i'(\xh(T))(\xi(T)+B(T)h)\leq 0,\,\mr{for}\,\,
i=0,\hdots,d_{\varphi},\\
&\eta_j'(\xh(T))(\xi(T)+B(T)h)=0,\,\mr{for}\,\, j=1,\hdots,d_{\eta}.
\end{split}
\ee
Define the set of \textit{transformed critical directions}
\benl
\label{chap1primc2}
\mathcal{P}:=
\left\{
\begin{split}
&(\xi,y,h)\in \X\times\Y\times \cR^m:\dot{y}_i=0\ \mathrm{over}\
I^i_0,\ y(0)=0,\ h:=y(T),\\
&\eqref{chap1xieq}\ \mr{and}\ \eqref{chap1translinearcons}\ \mr{hold}
\\
\end{split}
\right\}.
\eenl
Observe that for every $(\xi,y,h)\in \P$ and $1\leq i\leq m,$
\be
\label{chap1yconstant}
y_i\ \mr{is}\ \mr{constant}\ \mr{over}\ \mr{each}\ \mr{connected}\ \mr{component}\ \mr{of}\ I_0^i,
\ee
and at the endpoints the following conditions hold
\be
\label{chap1yi0}
\begin{split}
&y_i=0\ \mr{on}\
[0,d_1^i),\ \mr{if}\ 0\in I_0^i,\ \mr{and}\\
&y_i=h_i\ \mr{on}\
(c_{N_i}^i,T],\ \mr{if}\ T\in I^i_0,
\end{split}
\ee
where $c_1^i$ and $d_1^i$ were introduced in Assumption 2.
Define the set
 \benl
\label{chap1pc2}
\P_2:= \left\{
(\xi,y,h)\in \X_2\times\Uspace_2\times \cR^m:
\eqref{chap1xieq},\,\eqref{chap1translinearcons},\,\eqref{chap1yconstant}\ \mr{and}\ \eqref{chap1yi0}\ \mr{hold}
\right\}.
\eenl

\begin{lemma}\label{chap1DenseC2}
 $\P$ is a dense subset of $\P_2$ in the $\X_2\times \Uspace_2\times
\cR^m-$topology.
\end{lemma}
\begin{proof} The inclusion is immediate. In order to prove the density, consider the following sets.
\benl
X:=\{(\xi,y,h)\in \X_2\times\Uspace_2\times \cR^m: \eqref{chap1xieq},\, \eqref{chap1yconstant}\ \mr{and}\ \eqref{chap1yi0}\ \mr{hold}\},
\eenl
\benl
L:=\{(\xi,y,y(T))\in
\X\times\Y\times \cR^m:\, y(0)=0,\ \eqref{chap1xieq}\ \mr{and}\ \eqref{chap1yconstant}\ \mr{hold}\}, \eenl \benl
C:=
\left\{
(\xi,y,h)\in X:\ \eqref{chap1translinearcons}\ \mr{holds}
\right\}.
\eenl
By Lemma \ref{chap1PdenseHadd}, $L$ is a dense subset of $X.$ The conclusion follows with Lemma \ref{chap1lemmadense}.
\end{proof}

\subsection{Transformed second variation}
We are interested in
writing $\Omega$ in terms of variables $y$ and $\xi$ defined in \eqref{chap1yxi}. Introduce
the following notation for the sake of simplifying the presentation.
\begin{definition}
 Consider the following matrices of sizes $n\times n,m\times
n$ and $m\times n,$ respectively.
\be
\label{chap1QCM}
Q\lam:=H_{xx}\lam,\quad C\lam:=H_{ux}\lam, \quad M\lam:=B^\top Q\lam- \dot C\lam-C\lam A,
\ee
where $A$ and $B$ were defined in \eqref{chap1AB}. Notice that $M$ is well-defined as $C$ is Lipschitz-continuous on $t.$
 Decompose matrix $C\lam B$ into its symmetric
and skew-symmetric parts, i.e.  consider
\be
\label{chap1SV}
S\lam:=\half(C\lam B+(C\lam B)^\top),\quad V\lam:=\half(C\lam B-(C\lam B)^\top).
\ee
\end{definition}

\begin{remark}
Observe that, since $C\lam$ and $B$ are Lipschitz-continuous, $S\lam$ and $V\lam$ are Lipschitz-continuous as well. In fact, simple computations yield
\be
\label{chap1Vind}
S_{ij}\lam=\half\psi(f_i'f_j+f_j'f_i),\quad V_{ij}\lam=\half\psi[f_i,f_j],\quad \mr{for}\ i,j=1,\hdots,m,
\ee
where
\be
[f_i,f_j]:=f_i'f_j-f_j'f_i.
\ee
\end{remark}

With this notation, $\Omega$ takes the form
\benl
\Omega\lam(\delta x,v)=\half \ell''\lam (\xh(T))( \delta x(T))^2
+\half\int_0^T[(Q\lam \delta x,\delta x)+2 (C\lam\delta x,v)]\dtt.
\eenl
Define the $m\times m$ matrix
\be
\label{chap1R}
R\lam:=B^\top Q\lam B-C\lam B_1-(C\lam B_1)^\top-\dot S\lam,
\ee
where $B_1$ was introduced in equation \eqref{chap1B1}. Consider
the function $g\lam$ from $\cR^n\times \cR^m$ to $\cR$ defined
by:
\be
\label{chap1g}
g\lam(\zeta,h):=
\half\ell''\lam(\xh(T))(\zeta+B(T)h)^2
+
\half(C\lam(T)(2\zeta+B(T)h),h).
\ee

\begin{remark}
\begin{itemize}
\item[(i)] We use the same notation for the matrices $Q\lam,$ $C\lam,$
$M\lam,$ $\ell''\lam(\xh(T))$ and for the
bilinear mapping they define.
\item[(ii)] Observe that when $m=1,$  the function $V\lam \equiv 0$ since it becomes a skew-symmetric scalar.
\end{itemize}
\end{remark}

\begin{definition}
 Define the mapping over $\X\times\Y\times\U$ given by
\be
\label{chap1OmegaG}
\begin{split}
\Omega_{\P}&[\lambda](\xi,y,v):=g[\lambda](\xi(T),y(T))\\
&+\int_0^T\{\half(Q[\lambda]\xi,\xi)+2(M[\lambda]\xi,y)
+\half(R[\lambda]y,y)+(V\lam y,v)\}\dtt,
\end{split}
\ee
with $g\lam,$ $Q\lam,$ $M\lam,$ $R\lam$ and $V\lam$ defined in
\eqref{chap1QCM}-\eqref{chap1g}.
\end{definition}

The following theorem shows that $\Omega_{\P}$ coincides with $\Omega.$ See e.g. \cite{Dmi87}.

\begin{theorem}
\label{chap1secondvariation}
Let $(z,v)\in \W$ satisfying \eqref{chap1lineareq} and $(\xi,y)$ be defined by \eqref{chap1yxi}. Then
\benl
\label{chap1Omeganew}
\Omega[\lambda](z,v)=\Omega_{\P}\lam (\xi,y,v).
\eenl
\end{theorem}

\begin{proof}
We omit the dependence on $\lambda$ for the sake of
simplicity. Replace $z$ by its expression
in \eqref{chap1yxi} and obtain
\be
\begin {split}
\Omega&(z,v)
=
\half\ell''(\xh(T))(\xi(T)+B(T)y(T))^2\\
&+
\half \int_0^T [(Q(\xi+By),\xi+By)+(C(\xi+By),v)+(C^\top v,\xi+By) ]\dtt.
\label{chap1J2}
\end{split}
\ee
Integrating by parts yields
\be\label{chap1J2.2.1}
\int_0^T ( C\xi,v) \dtt
=[( C\xi,y)
]_0^T-\int_0^T(\dot{C}\xi+C(A\xi+B_1y),y)\dtt,
\ee
and
\be\label{chap1J2.2.2}
\begin{split}
\int_0^T ( CBy,v)\dtt
&=\int_0^T ( (S+V)y,v)\dtt\\
&=\half[( Sy,y)]_0^T+\int_0^T(-\half(\dot{S}y,y)+(
Vy,v))\dtt.
\end{split}
\ee
Combining \eqref{chap1J2}, \eqref{chap1J2.2.1} and \eqref{chap1J2.2.2} we get the desired result.
\end{proof}

\begin{corollary}
\label{chap1Omegas}
 If $V[\lambda]\equiv 0$ then $\Omega$ does not involve $v$ explicitly, and it can be
expressed in terms of $(\xi,y,y(T)).$
\end{corollary}

In view of \eqref{chap1Vind}, the previous corollary holds in particular if $[f_i,f_j]=0$ on the reference trajectory
for each pair $1\leq i\mi j\leq m.$

\begin{corollary}\label{chap1CN2vector}
If $\wh$ is a weak minimum, then
 \benl
\max_{\lambda\in\Lambda}\Omega_{\P}[\lambda](\xi,y,v)\geq 0,
\eenl
for every $(z,v)\in \C_2$
and $(\xi,y)$ defined by \eqref{chap1yxi}.
\end{corollary}

\subsection{New second order condition}
In this section we present a necessary condition involving the variable $(\xi,y,h)$ in $\P_2.$ To achieve this we remove the explicit dependence on $v$ from the second variation, for certain subset of multipliers. Recall that we consider $\lambda=(\alpha,\beta)$ as elements of $\cR^s.$

\begin{definition}
Given $M\subset \cR^s,$  define
\benl
G(M):=
\{\lambda\in M:V_{ij}\lam(t)=0\ \mr{on}\ I^{i}_+\cap I^j_+,\ \mr{for}\ \mr{any}\ \mr{pair}\  1\leq i\mi j\leq m\}.
\eenl
\end{definition}

\begin{theorem}\label{chap1goh}
Let $M\subset \cR^s$ be convex and compact, and
assume that
\be
\label{chap1maxM}
\max_{\lambda\in M}\Omega_{\P}\lam(\xi,y,\dot y)\geq0,\quad \mr{for}\ \mr{all}\ \ (\xi,y,h)\in \P.
\ee
Then
\benl
\max_{\lambda\in G(M)}\Omega_{\P}\lam(\xi,y,\dot y)\geq0,\quad\mr{for}\ \mr{all}\ \ (\xi,y,h)\in \P.
\eenl
\end{theorem}

The proof is based on some techniques introduced in Dmitruk \cite{Dmi77,Dmi87} for the proof of similar theorems.

Let $1\leq i \mi j \leq m$ and $t^*\in \mr{int}\ I^i_+\cap I^j_+.$ Take $y\in \Y$ satisfying
\be\label{chap1condyr}
y(0)=y(T)=0,\quad y_k=0,\ \mr{for}\ \  k\neq i,k\neq j.
\ee
Such functions define a linear continuous mapping $\mathpzc{r}:\cR^{s,*}\rightarrow \cR$ by
\be
\label{chap1defr}
\lambda\mapsto \mathpzc{r}\lam:=\intT (V\lam(t^*)y,\dot y)\dtt.
\ee
By condition \eqref{chap1condyr}, and since $V\lam$ is skew-symmetric,
\benl
\intT (V\lam(t^*)y,\dot y)\dtt=
V_{ij}\lam(t^*)\intT (y_i\dot{y}_j-y_j\dot{y}_i)\dtt.
\eenl
Each $\mathpzc{r}$ is an element of the dual space of $\cR^{s,*},$ and it can thus be identified with an element of $\cR^s.$ Consequently, the subset of $\cR^s$ defined by
\benl
R_{ij}(t^*):=\{\mathpzc{r}\in\cR^s:y\in\Y\ \mr{satisfies}\ \eqref{chap1condyr},\ \mathpzc{r}\ \mr{is}\ \mr{defined}\ \mr{by}\  \eqref{chap1defr}\},
\eenl
is a linear subspace of $\cR^s.$ Now, consider all the finite collections
\benl
\Theta_{ij}:=\left\{\theta=\{t^1 < \cdots <t^{N_{\theta}}\}:t^k\in \mr{int}\, I^i_+\cap I^j_+\ \mr{for}\ k=1,\hdots,N_{\theta}\right\}.
\eenl
Define
\benl
\R:=\sum_{i\mi j} \bigcup_{\theta\in \Theta_{ij}} \sum_{k=1}^{N_{\theta}} R_{ij}(t^k).
\eenl
Note that $\R$ is a linear subspace of $\cR^s.$ Given $(\xi,y,y(T))\in\P,$ let the mapping $\mathpzc{p}_{y}:\cR^{s,*}\rightarrow \cR$ be given by
\be
\label{chap1plam}
\lambda\mapsto \mathpzc{p}_y\lam:= \Omega_{\P}\lam(\xi,y,\dot{y}).
\ee
Thus, $\mathpzc{p}_y$ is an element of $\cR^{s}.$

\begin{lemma}
\label{chap1approxlemma}
Let $(\bar\xi,\yb,\yb(T))\in \P$ and $\mathpzc{r}\in \R.$ Then there exists a sequence
\\$\{(\xi^{\nu},y^{\nu},y^{\nu}(T))\}$ in $\P$ such that
\be
\label{chap1limitOmegaeq}
\Omega_{\P}\lam (\xi^{\nu},y^{\nu},\dot y^{\nu})\longrightarrow \mathpzc{p}_{\yb}\lam+\mathpzc{r}\lam.
\ee
\end{lemma}

\begin{proof}
Take $(\bar\xi,\yb,\yb(T))\in \P,$ its corresponding critical direction $(\zb,\vb)\in\C$ related via \eqref{chap1yxi} and $\mathpzc p_{\yb}$ defined in \eqref{chap1plam}.
Assume that $\mathpzc{r}\in R_{ij}(t^*)$ for some  $1\leq i\mi j\leq m$ and $t^*\in\mr{int}\ I^i_+\cap I^j_+,$ i.e.   $\mathpzc{r}$
is associated via \eqref{chap1defr} to some function $\tilde{y}$  verifying \eqref{chap1condyr}.
Take $\tilde{y}(t)=0$ when $t\notin[0,T].$  Consider
\be
\label{chap1tildey}
\tilde{y}^{\nu}(t):=\tilde{y}(\nu(t-t^*)),\quad
\breve{y}^{\nu}:=\yb+\tilde{y}^{\nu}.
\ee
Let $\breve{\xi}^{\nu}$ be the solution of \eqref{chap1xieq} corresponding to $\breve{y}^{\nu}.$
Observe that for large enough ${\nu},$ as $t^*\in \mr{int}\ I^i_+\cap I^j_+,$
\be
\label{chap1brevey}
\dot{\breve{y}}^{\nu}_{k}=0,\ \mr{a.e.}\  \mr{on}\ I_0^k,\ \mr{for}\ k=1,\hdots,m.
\ee
Let $(\tilde z^{\nu},\tilde v^{\nu})$ and $(\breve z^{\nu},\breve v^{\nu})$ be the points associated by transformation \eqref{chap1yxi} with $(\tilde\xi^{\nu},\tilde y^{\nu},\tilde y^{\nu}(T))$ and $(\breve\xi^{\nu},\breve y^{\nu},\breve y^{\nu}(T)),$ respectively. By \eqref{chap1brevey}, we get
\benl
\label{chap1brevev}
\breve{v}^{\nu}_{k}=0,\ \mr{a.e.}\  \mr{on}\ I_0^k,\ \mr{for}\ k=1,\hdots,m.
\eenl
 Note, however, that $(\breve z^{\nu},\breve v^{\nu})$ can violate the terminal constraints defining $\C_{\infty},$ i.e. the constraints defined in \eqref{chap1linearcons}. Let us look for an estimate of the magnitude of this violation. Since
\be
\label{chap1tildey0}
\|\tilde y^{\nu}\|_{1}=\mathcal{O}(1/{\nu}),
\ee
and  $(\tilde \xi^{\nu},\tilde y^{\nu})$ is solution of \eqref{chap1xieq},  Gronwall's Lemma implies
\benl
\label{chap1tildexiqT}
|\tilde \xi^{\nu}(T)|=\mathcal{O}(1/{\nu}).
\eenl
On the other hand, notice that $\breve z^{\nu}(T)=\zb(T)+\tilde\xi^{\nu}(T),$ and thus
\benl
|\breve z^{\nu}(T)-\zb(T)|=\mathcal{O}(1/{\nu}).
\eenl
By Hoffman's Lemma (see \cite{Hof52}), there exists  $(\Delta z^{\nu},\Delta v^{\nu})\in\W$ satisfying
$ \|\Delta v^{\nu}\|_{\infty}+
\|\Delta z^{\nu}\|_{\infty}=\mathcal{O}(1/\nu),
$
and such that $(z^{\nu},v^{\nu}):=(\breve z^{\nu},\breve v^{\nu})+(\Delta z^{\nu},\Delta v^{\nu})$ belongs to $\C_{\infty}.$ Let $({\xi}^{\nu},{y}^{\nu},{y}^{\nu}(T))\in\P$ be defined by \eqref{chap1yxi}.
Let us show that for each $\lambda\in M,$
\benl
\lim_{{\nu}\rightarrow\infty}\Omega_{\P}\lam(\xi^{\nu},y^{\nu},\dot{y}^{\nu})= \mathpzc p_{\yb}\lam+\mathpzc r\lam.
\eenl
Observe that
\be\label{chap1OmegaPminusp}
\lim_{\nu\rightarrow\infty} \Omega_{\P}\lam(\xi^{\nu},y^{\nu},\dot{y}^{\nu})- \mathpzc p_{\yb}\lam=
\lim_{\nu\rightarrow\infty} \int_0^T
\{(V[\lambda]\yb,\dot{\tilde y}^{\nu})+(V[\lambda]\tilde y^{\nu},\dot{\tilde y}^{\nu})\}\dtt,
\ee
since the terms involving $\xi^{\nu}-\bar \xi,$ $y^{\nu}-\bar y$ or $\Delta v^{\nu}$
vanish as $\|\xi^{\nu}-\bar \xi\|_{\infty}\rightarrow 0$ and $\|y^{\nu}-\bar y\|_1\rightarrow 0.$
Integrating by parts the first term in the right hand-side of \eqref{chap1OmegaPminusp}. we obtain
\benl
\int_0^T (V[\lambda]\yb,\dot{\tilde y}^{\nu})\dtt
=[(V[\lambda]\yb,\tilde y^{\nu})]_0^T-\int_0^T \{(\dot{V}[\lambda]\yb,\tilde y^{\nu})+(V[\lambda]\dot{\yb},\tilde y^{\nu})\}\dtt\overset{\nu\rightarrow \infty}\rightarrow 0,
\eenl
by \eqref{chap1tildey0} and since $\tilde{y}^{\nu}(0)=\tilde{y}^{\nu}(T)=0.$
Coming back to \eqref{chap1OmegaPminusp} we have
\begin{align*}
\lim_{\nu\rightarrow\infty} \Omega_{\P}\lam&(\xi^{\nu},y^{\nu},\dot{y}^{\nu})- \mathpzc p_{\yb}\lam =\lim_{\nu\rightarrow\infty} \int_0^T (V[\lambda]\tilde y^{\nu},\dot{\tilde y}^{\nu})\dtt\\
&=\lim_{\nu\rightarrow\infty}\intT (V[\lambda](t)\tilde y(\nu(t-t^*)),\dot{\tilde y}(\nu(t-t^*)))\mr{d}\nu t\\
&=\lim_{\nu\rightarrow\infty}\int_{-\nu t^*}^{\nu(T-t^*)} (V[\lambda](t^*+s/\nu)\tilde y(s),\dot{\tilde y}(s))\mr{d}s
= \mathpzc r[\lambda],
\end{align*}
and thus \eqref{chap1limitOmegaeq} holds when $\mathpzc{r}\in R_{ij}(t^*).$

Consider the general case when $ \mathpzc{r}\in \R,$ i.e.
$\mathpzc{r}=\ds\sum_{i\mi j}\sum_{k=1}^{N_{ij}}\mathpzc{r}_{ij}^k,$  with each $\mathpzc{r}_{ij}^k$ in $R_{ij}(t_{ij}^k).$
Let $\tilde{y}_{ij}^k$ be associated with  $\mathpzc{r}_{ij}^k$ by \eqref{chap1defr}. Define $\tilde{y}^{k,\nu}_{ij}$ as  in \eqref{chap1tildey}, and follow the previous procedure for $\yb+ \ds\sum_{i\mi j}\sum_{k=1}^{N_{ij}} \tilde{y}^{k,\nu}_{ij}$  to get the desired result.

\end{proof}

\if{
\begin{definition}
Given $M\subset \cR^s$ a convex compact set, a pair $1\leq i\neq j\leq m,$ and an interval $(a,b)\subset I^{i}_+\cap I^j_+,$ define
\be
G_{(a,b)}^{ij}(M):=\{\lambda\in M:V_{ij}\lam \equiv 0\ \mr{on}\ (a,b)\}.
\ee
\end{definition}

\begin{lemma} Let $M\subset \cR^s$ be convex and compact, $1\leq i\neq j\leq m$ a pair of indexes and $(a,b)\subset I^{i}_+\cap I^j_+.$
Assume that \eqref{chap1maxM} holds.
Then
\be\label{chap1maxGij}
\max_{\lambda\in G_{(a,b)}^{ij}(M)}\Omega_{\P}\lam(\xi,y,\dot y)\geq0,\quad \mr{for}\ \mr{all}\ \ (\xi,y,y(T))\in \P.
\ee
\end{lemma}
}\fi

\begin{proof}\textbf{[of Theorem \ref{chap1goh}]}
Take $(\bar\xi,\yb,\yb(T))\in\P$ and $r\in\R.$ By Lemma \ref{chap1approxlemma} there exists a sequence $\{(\xi^{\nu},y^{\nu},y^{\nu}(T))\}$ in $\P$ such that for each $\lambda\in M,$
\benl
\label{chap1limOmeganu}
\Omega_{\P}\lam(\xi^{\nu},y^{\nu},\dot{y}^{\nu})
\rightarrow
\Omega_{\P}\lam(\bar\xi,\yb,\dot\yb)+r\lam.
\eenl
Since this convergence is uniform over $M,$ from \eqref{chap1maxM} we get that
\benl
\label{chap1maxomegar}
\max_{\lambda\in M}(\Omega_{\P}\lam(\bar\xi,\yb,\dot{\yb})+r\lam)\geq 0,\quad \mr{for}\ \mr{all}\ \mathpzc{r}\in\R.
\eenl
Hence
\be
\label{chap1infmax}
\inf_{r\in\R}\max_{\lambda\in M}(\Omega_{\P}\lam(\bar\xi,\yb,\dot{\yb})+r\lam)\geq 0,
\ee
where the expression in brackets is linear both in $\lambda$ and $r.$ Furthermore, note that $M$ and $\R$ are convex, and $M$ is compact. In light of MinMax Theorem \cite[Corollary 37.3.2, page 39]{Roc70} we can invert the order of $\inf$ and $\max$ in \eqref{chap1infmax} and obtain
\be
\label{chap1maxinf}
\max_{\lambda\in M}\inf_{r\in\R}(\Omega_{\P}\lam(\bar\xi,\yb,\dot{\yb})+r\lam)\geq 0.
\ee
Suppose that, for certain $\lambda\in M,$ there exists $r\in\R$ with $r\lam\neq 0.$ Then the infimum in \eqref{chap1maxinf} is  $-\infty$ since $\R$ is a linear subspace.
Hence, this $\lambda$ does not provide the maximal value of the infima, and so, we can restrict the maximization to the set of $\lambda\in M$ for which $r\lam=0$ for every $r\in \R.$ Note that this set is $G(M),$ and thus the conclusion follows.
\end{proof}

\if{
Observe that Assumption 2 implies that
$I^{i}_+\cap I^{j}_+$ can be expressed as a finite union of intervals. We will say that $(a,b)\subset I^{i}_+\cap I^{j}_+$ is a \textit{maximal composing interval} if there does not exist $(c,d)\subset I^{i}_+\cap I^{j}_+$ containing $(a,b)$ with strict inclusion.

\begin{proof}
Consider two arbitrary pairs $i\neq j,$ and $i'\neq j'.$
Let $(a,b)$ and $(a',b')$ be maximal composing intervals of $I^{i}_+\cap I^{j}_+$ and $I^{i'}_+\cap I^{j'}_+,$ respectively.
Observe that with the same arguments used in the proof of previous Lemma we can show that
\be\label{chap1maxinter}
\max_{\lambda\in G_{(a,b)}^{ij}(M)\cap G_{(a',b')}^{i'j'}(M) }\Omega_{\P}\lam(\bar\xi,\yb,\dot{\yb})\geq 0.
\ee
Furthermore, we can iteratively restrinct the maximum in \eqref{chap1maxinter} to a smaller subset of multipliers obtained intersecting the current subset with $G_{(a'',b'')}^{i''j''}(M).$ In a finite number of iterations,
the desired result is achieved.
\end{proof}
}\fi

Consider for $i,j=1,\hdots,m:$
\benl
\label{chap1Iij}
I_{ij}:=\{t\in(0,T):\uh_i(t)=0,\ \uh_j(t)>0\}.
\eenl
By Assumption 2, $I_{ij}$ can be expressed as a finite union of intervals, i.e.
\benl
\label{chap1Iijk}
I_{ij}=\bigcup_{k=1}^{K_{ij}} I_{ij}^k,\ \mr{where} \
I_{ij}^k:=(c_{ij}^k,d_{ij}^k).
\eenl
Let $(z,v)\in \C_{\infty},$ $i\neq j,$ and $y$ be defined by \eqref{chap1yxi}. Notice that $y_i$ is constant on each $(c_{ij}^k,d_{ij}^k).$ Denote with $y_{i,j}^k$ its value on this interval.

\begin{proposition}
\label{chap1cnvector}
Let $(z,v)\in \C_{\infty},$ $y$ be defined by \eqref{chap1yxi}
and $\lambda\in G(\Lambda).$ Then
\benl
\label{chap1expansionV}
\int_0^T (V[\lambda]y,v)\dtt=
\sum_{\substack{i\neq j\\i,j=1}}^{m}
\sum_{k=1}^{K_{ij}}y^{k}_{i,j}
\left\{
[V_{ij}[\lambda]y_j]_{c^{k}_{ij}}^{d^{k}_{ij}}-\int_{c^{k}_{ij}}^{d^{k}_{ij}}\dot{V}_{ij}[\lambda]y_j\dtt
\right\}.
\eenl
\end{proposition}

\begin{proof}
Observe that
\be
\label{chap1sumVij}
\int_0^T (V\lam y,v)\dtt=
\sum_{\substack{i\neq j\\i,j=1}}^m\int_0^TV_{ij}\lam y_iv_j\dtt,
\ee
since $V_{ii}\lam\equiv0.$ Fix $i\neq j,$ and recall that that $V_{ij}\lam$ is differentiable in time (see expression \eqref{chap1Vind}). Since $(z,v)\in \C_{\infty}$ and $\lambda\in G(\Lambda),$
\be
\label{chap1sumk}
\begin{split}
 \int_0^T V_{ij}\lam y_iv_j\dtt
&=\int_{I_{ij}}V_{ij}\lam y_iv_j\dtt
=\sum_{k=1}^{K_{ij}} \int_{c_{ij}^k}^{d_{ij}^k}
V_{ij}\lam y_iv_j\dtt\\
&= \sum_{k=1}^{K_{ij}} y_{i,j}^k
\left\{ [V_{ij}\lam y_j]_{c_{ij}^k}^{d_{ij}^k}
-\int_{c_{ij}^k}^{d_{ij}^k} \dot{V}_{ij}\lam y_j\dtt\right\},
\end{split}
\ee
where the last equality was obtained by integrating by parts and knowing that $y_i$ is constant on $I_{ij}.$ The desired result follows from \eqref{chap1sumVij} and \eqref{chap1sumk}.
\end{proof}

Given a real function $h$ and $c\in \cR,$ define
\benl
h(c+):=\lim_{t\rightarrow c+} h(t),\ \mr{and}\
h(c-):=\lim_{t\rightarrow c-} h(t).
\eenl
\begin{definition} Let $(\xi,y,h)\in\P_2$ and $\lambda\in G(\Lambda).$ Define
\benl
\begin{split}
&\Xi\lam(\xi,y,h):=\\
&2\sum_{\substack{i\neq j\\i,j=1}}^{m}
\sum_{\substack{k=1\\c^{k}_{ij}\neq 0}}^{K_{ij}}
y^{k}_{i,j}
\left\{
V_{ij}\lam(d^{k}_{ij})y_j(d^{k}_{ij}+)
-V_{ij}\lam (c^{k}_{ij}) y_j(c^{k}_{ij}-)
-\int_{c^{k}_{ij}}^{d^{k}_{ij}}\dot{V}_{ij}\lam
y_j\dtt\right\},
\end{split}
\eenl
where the above expression is interpreted as follows:
\begin{itemize}
\item[(i)] $y_j(d^{k}_{ij}+):=h_j,$ if $d^{k}_{ij}=T,$
\item[(ii)]$V_{ij}\lam (c^{k}_{ij})y_j(c^{k}_{ij}-):=0,$ if $\uh_i\gr 0$ and $\uh_j\gr 0$ for $t<c^{k}_{ij},$
\item[(iii)]$V_{ij}\lam(d^{k}_{ij})y_j(d^{k}_{ij}+):=0,$ if $\uh_i\gr 0$ and
$\uh_j\gr 0$ for $t>d^{k}_{ij}.$
\end{itemize}
\end{definition}
\begin{proposition}
\label{chap1Vquad}
The following properties for $\Xi$ hold.
\begin{itemize}
\item [(i)]$\Xi\lam(\xi,y,h)$ is well-defined for each
$(\xi,y,h)\in \P_2,$ and $\lambda\in G(\Lambda).$
\item [(ii)] If $\{(\xi^{\nu},y^{\nu},y^{\nu}(T))\}\subset \P$
converges in the
$\X_2\times\Uspace_2\times
\cR^m-$ topology to $(\xi,y,h)\in \P_2$ , then
\benl
\int_0^T  (V\lam y^{\nu},\dot{y}^{\nu})\dtt
\overset{\nu\rightarrow \infty}\longrightarrow
\Xi\lam(\xi,y,h).
\eenl
\end{itemize}
\end{proposition}

\begin{proof}
\textbf{(i)}
Take $(\xi,y,h)\in \P_2.$ First observe that $y_i\equiv y_{i,j}^k$ over $(c^{k}_{ij},d^{k}_{ij}).$
As $c^{k}_{ij}\neq 0,$ two possible situations can arise,
\begin{itemize}
 \item [(a)] for $t<c^{k}_{ij}:$ $\uh_j=0,$ thus $y_j$ is
constant, and consequently $y_j(c^{k}_{ij}-)$ is
well-defined,
\item [(b)] for $t<c^{k}_{ij}: $ $\uh_i>0$ and $\uh_j>0,$ thus
$V_{ij}\lam(c^{k}_{ij})=0$ since $\lambda\in G(\Lambda).$
\end{itemize}
The same analysis can be done for $t>d^{k}_{ij}$ when
$d^{k}_{ij}\neq T.$ We conclude that $\Xi$ is correctly defined.

\noindent \textbf{(ii)} Observe that since $y^{\nu}$ converges to $y$ in
the $\Uspace_2-$topology and since $y_i^{\nu}$ is constant over
$I_{ij},$ then $y_i$ is constant as well, and $y_i^{\nu}$
goes to $y_i$ pointwise on $I_{ij}.$ Thus, $
y_i^{\nu}(c^{k}_{ij})\longrightarrow y^{k}_{i,j},
$ and
$
y_i^{\nu}(d^{k}_{ij})\longrightarrow y^{k}_{i,j}.$
Now, for the terms on $y_j,$ the same analysis can be made, which yields either
$y_j^{\nu}(c^{k}_{ij})\longrightarrow y_j(c^{k}_{ij}-)$ or
$V_{ij}\lam (c^{k}_{ij})=0;$ and, either
$y_j^{\nu}(d^{k}_{ij})\longrightarrow y_j(d^{k}_{ij}+)$ or
$V_{ij}\lam (d^{k}_{ij})=0,$ when $d^{k}_{ij}<T.$ For
$d^{k}_{ij}=T,$  $y_j^{\nu}(T)\longrightarrow h_j$ holds.
\end{proof}

\begin{definition} For $(\xi,y,h)\in \P_2$ and $\lambda\in G(\Lambda)$ define
\benl
\begin{split}
\Omega_{\P_2}\lam(\xi,y,h):=&g\lam (\xi(T),h) +\Xi\lam(\xi,y,h)\\
&+\intT ((Q\lam\xi,\xi)+2(M\lam \xi,y)+(R\lam y,y))\dtt.
\end{split}
\eenl
\end{definition}
\begin{remark}\label{chap1OmegasRemark}
 Observe that when $m=1,$ the mapping $\Xi\equiv0$ since $V\equiv 0.$ Thus, in this case, $\Omega_{\P_2}$ can be defined for any element $(\xi,y,h)\in \X_2\times\U_2\times \cR$ and any $\lambda\in\Lambda.$ If we take $(z,v)\in \W$ satisfying \eqref{chap1lineareq}, and define $(\xi,y)$ by \eqref{chap1yxi}, then
\benl
\Omega\lam(z,v)=\Omega_{\P}\lam(\xi,y,\dot y)
=\Omega_{\P_2}\lam(\xi,y,y(T)).
\eenl
For $m\gr 1,$ the previous equality holds for $(z,v)\in\C_{\infty}.$
\end{remark}

\begin{lemma}\label{chap1limitOmega} Let  $\{(\xi^{\nu},y^{\nu},y^{\nu}(T)\}\subset \P$ be a sequence converging to  $(\xi,y,h)\in \P_2$ in the $\X_2\times \U_2\times \cR^m-$topology. Then
\benl
\lim_{\nu\rightarrow \infty}\Omega_{\P}\lam(\xi^{\nu},y^{\nu},\dot{y}^{\nu})
=\Omega_{\P_2}\lam(\xi,y,h).
\eenl
\end{lemma}

\if{
\begin{definition} Consider the subset of Lagrange multipliers
 \be
\Lambda_1:=\{\lambda\in\Lambda:\ \sum_{i=0}^{d_{\varphi}} \alpha_i=1\}.
\ee
\end{definition}

\begin{lemma}
 If \eqref{chap1cq} holds, then $\Lambda_1$
is compact and convex.
\end{lemma}
\begin{proof}
 Convexity is immediate. In order to show compactness, suppose on the contrary that there exists a sequence $(\lambda^k)\subset \Lambda_1$ with $\lambda^k=(\alpha^k,\beta^k,\psi^k),$ such that
$|\beta^k|\rightarrow \infty.$
Note that $(\lambda^k/|\beta^k|)\subset \Lambda$ is a bounded sequence, and hence it contains a converging subsequence. Denote with $\bar\lambda=(\bar\alpha,\bar\beta,\bar\psi)$ its limit. Thus, $\bar\lambda\in\Lambda,$ $\bar\alpha=0$ and $|\bar\beta|=1,$ contradicting Remark \ref{chap1remarkLambda1}.
\end{proof}
}\fi

Denote with $\mr{co}\,\Lambda$ the convex hull of $\Lambda.$

\begin{theorem}\label{chap1NCnew}
Let  $\wh$ be a weak minimum, then
\be
\label{chap1NCcoLambda}
\max_{\lambda\in G(\mr{co}\,\Lambda)}
\Omega_{\P_2}\lam(\xi,y,h)
\geq 0,\quad \mr{for}\ \mr{all}\ (\xi,y,h)\in \P_2.
\ee
\end{theorem}

\begin{proof}
Corollary \ref{chap1CN2vector} together with Theorem \ref{chap1goh} applied to $M:=\mr{co}\,\Lambda$ yield
\benl
\max_{\lambda\in G(\mr{co}\,\Lambda)}
\Omega_{\P}\lam(\xi,y,\dot y)
\geq 0,\quad \mr{for}\ \mr{all}\ (\xi,y,y(T))\in \P.
\eenl
The result follows from Lemma \ref{chap1DenseC2} and Lemma \ref{chap1limitOmega}.
\end{proof}

\begin{remark}\label{chap1remarkLambda1}
Notice that in case \eqref{chap1cq} is not satisfied, condition \eqref{chap1NCcoLambda} does not provide any useful information as $0\in\mr{co}\,\Lambda.$ On the other hand, if \eqref{chap1cq} holds, every $\lambda= (\alpha,\beta,\psi)\in\Lambda$ necessarily has $\alpha\neq0,$ and thus $0\notin \mr{co}\,\Lambda.$
\end{remark}


\section{Sufficient condition}

Consider the problem for a scalar control, i.e.  let $m=1.$ This section provides a sufficient condition for  Pontryagin
optimality.
\begin{definition} \label{chap1defgamma}
Given $(y,h)\in\U_2\times \cR,$ let
\benl\label{chap1gamma}
\gamma(y,h):=\intT y(t)^2\dtt+|h|^2.
\eenl
\end{definition}

\begin{definition}
A sequence $\{v_k\}\subset \Uspace$ \textit{converges to 0 in the Pontryagin sense} if $\|v_k\|_1\rightarrow 0$ and there exists $N$ such that $\|v_k\|_{\infty}<N.$
\end{definition}
\begin{definition}
We say that $\wh$ satisfies  \textit{$\gamma-$quadratic growth condition in the Pontryagin sense} if there exists $\rho>0$ such that, for every sequence of feasible variations $\{(\delta x_k,v_k)\}$ with
$\{v_k\}$ converging to 0 in the Pontryagin sense,
\be\label{chap1qgdef}
J(\uh+v_k)-J(\uh)\geq \rho\gamma(y_k,y_k(T)),
\ee
holds for a large enough $k,$ where $y_k$ is defined by \eqref{chap1yxi}. Equivalently, for all $N\gr 0,$ there exists $\varepsilon\gr 0$ such that if $\|v\|_{\infty}\mi N$ and $\|v\|_{1}\mi\varepsilon,$ then
\eqref{chap1qgdef} holds.
\end{definition}

\begin{definition}
 We say that $\wh$ is \textit{normal} if $\alpha_0\gr 0$ for every $\lambda\in\Lambda.$
\end{definition}

\begin{theorem}
\label{chap1sc2scalar}
Suppose that
there exists $\rho\gr 0$ such that
\be\label{chap1unifpos}
\max_{\lambda\in\Lambda}\Omega_{\P_2}\lam(\xi,y,h)\geq \rho\gamma(y,h),\quad \mr{for}\ \mr{all}\ (\xi,y,h)\in\P_2.\ee
Then $\wh$ is a Pontryagin minimum satisfying $\gamma-$ quadratic growth. Furthermore, if $\wh$ is normal, the converse holds.
\end{theorem}
\begin{remark}
 In case the bang arcs are absent, i.e.  the control is totally singular, this theorem reduces to one proved in Dmitruk \cite{Dmi83,Dmi87}.
\end{remark}

Recall that $\Phi$ is defined in \eqref{chap1lagrangian}.
We will use the following technical result.
\begin{lemma}
\label{chap1lemmasc2}
Consider $\{v_k\}\subset \U$ converging to 0 in the Pontryagin sense. Let $u_k:=\uh+v_k$ and let $x_k$ be the corresponding solution of equation \eqref{chap1stateeq}. Then for every $\lambda\in\Lambda,$
\be
\label{chap1taylor0}
\Phi[\lambda](x_k,u_k)
=\Phi[\lambda](\xh,\uh)+\int_0^T H_u\lam(t)v_k(t)\dtt+
\Omega[\lambda](z_k,v_k)+o(\gamma_k),
\ee
where $z_k$ is defined by \eqref{chap1lineareq}, $\gamma_k:=\gamma(y_k,y_k(T)),$ and $y_k$ is defined by \eqref{chap1yxi}.
\end{lemma}

\begin{proof}
By Lemma \ref{chap1expansionlagrangian} we can write
\benl
\label{chap1taylorlemma}
\Phi[\lambda](x_k,u_k)
=\Phi[\lambda](\xh,\uh)+\int_0^T H_u\lam(t)v_k(t)\dtt+
\Omega[\lambda](z_k,v_k)+R_k,
\eenl
where, in view of Lemma \ref{chap1lemmabound2},
\be
\label{chap1Rk}
 R_k:=
\Delta_k\Omega[\lambda]
+\int_0^T (H_{uxx}[\lambda](t)\delta x_k(t),\delta
x_k(t),v_k(t))\dtt + o(\gamma_k),
\ee
with $\delta x_k:=x_k-\xh,$ and
\be
\label{chap1DeltaOmega}
\Delta_k\Omega[\lambda]:= \Omega[\lambda](\delta
x_k,v_k)-\Omega[\lambda](z_k,v_k).
\ee
Next, we prove that
\begin{equation}
\label{chap1rest}
R_k=o(\gamma_k).
\end{equation}
Note  that $\Q(a,a)-\Q(b,b)=\Q(a+b,a-b),$ for any bilinear mapping $\Q,$ and any pair $a,b.$ Put
$\eta_k:=\delta x_k-z_k.$
Hence, from \eqref{chap1DeltaOmega}, we get
\begin{align*}
\Delta_k\Omega[\lambda]
=&\half\ell''\lam(\xh(T))(\delta x_k(T)+z_k(T),\eta_k(T))\\
&+\half\int_0^T(H_{xx}[\lambda](\delta x_k+z_k),\eta_k)\dtt
+\int_0^T (H_{ux}[\lambda]\eta_k,v_k)\dtt.
\end{align*}
By Lemmas \ref{chap1lemmabound2} and \ref{chap1etaPont} in the Appendix, the
first and the second terms are of order $o(\gamma_k).$  Integrate by parts the last
term to obtain
\begin{align}
\label{chap1deltaOmega2}
\int_0^T &(H_{ux}\lam\eta_k,v_k)\dtt\\
&=\left[(H_{ux}\lam\eta_k,y_k)\right]_0^T-\int_0^T \{(\dot
H_{ux}\lam\eta_k,y_k)+(H_{ux}\lam\dot\eta_k,y_k)\}\dtt.
\label{chap1dotetak}
\end{align}
Thus, by Lemma \ref{chap1etaPont} we deduce that the first two terms in \eqref{chap1dotetak} are of order $o(\gamma_k).$
It remains to deal with last term in the integral. Replace $\dot{\eta}_k$ by its expression in equation \eqref{chap1etadot} of Lemma \ref{chap1etaPont}:
\be
\label{chap1Huxetadot}
\begin{split}
\int_0^T (H_{ux}\lam\dot\eta_k,y_k)\dtt
&=\int_0^T (H_{ux}\lam\left(\sum_{i=0}^1 \uh_if_i'(\xh)\eta_k+v_kf_1'(\xh)\delta x_k+\zeta_k\right),y_k)\dtt\\
&=o(\gamma_k)+\int_0^T \ddt \left(\frac{y_k^2}{2}\right) H_{ux}\lam f_1'(\xh)\delta x_k\dtt,
\end{split}
\ee
where the second equality follows from Lemmas \ref{chap1lemmabound2} and \ref{chap1etaPont}. Integrating the last term by parts, we obtain
\be
\begin{split}\label{chap1Huxf1prime}
&\int_0^T \ddt \left(\frac{y_k^2}{2}\right) H_{ux}\lam
f_1'(\xh)\delta x_k\dtt
= \left[\frac{y_k^2}{2}H_{ux}\lam f_1'(\xh)\delta x_k\right]_0^T\\
&-\intT \frac{y_k^2}{2}\ddt \left(H_{ux}\lam f_1'(\xh)\right)\delta x_k\dtt
-\intT \frac{y_k^2}{2}H_{ux}\lam f_1'(\xh)\dot{\delta x_k}\dtt\\
&=o(\gamma_k)-\intT \ddt\left(\frac{y_k^3}{6}\right)H_{ux}\lam f_1'(\xh) f_1(\xh)\dtt \\
&=o(\gamma_k)-\left[\frac{y_k^3}{6} H_{ux}\lam f_1'(\xh) f_1(\xh)\right]_0^T
+\intT \frac{y_k^3}{6}\ddt\left(H_{ux}\lam f_1'(\xh) f_1(\xh)\right)\dtt\\
&=o(\gamma_k),
\end{split}
\ee
where we used Lemma \ref{chap1etaPont} and, in particular, equation \eqref{chap1deltaxdot}.
From \eqref{chap1Huxetadot} and \eqref{chap1Huxf1prime}, it follows that the term in \eqref{chap1deltaOmega2} is of order $o(\gamma_k).$ Thus,
\be\label{chap1DeltaOmega'}
\Delta_k\Omega[\lambda]\leq o({\gamma_k}).
\ee
Consider now the third order term in \eqref{chap1Rk}:
\be\label{chap1Huxx}
\begin{split}
&\intT (H_{uxx}\lam \delta x_k,\delta x_k,v_k)\dtt
=[y_k\delta x_k^\top H_{uxx}\lam \delta x_k]_0^T\\
&-\intT y_k\delta x_k^\top \dot{H}_{uxx}\lam\delta x_k\dtt
-2\intT y_k\delta x_k^\top H_{uxx}\lam\dot{\delta x_k}\dtt
\\
&=o(\gamma_k)-\intT \ddt(y_k^2) \delta x_k^\top H_{uxx}\lam f_1(\xh)\dtt \\
&=o(\gamma_k)-\left[y_k^2\delta x_k^\top H_{uxx}\lam f_1(\xh)\right]_0^T
-\intT y_k^2v_kf_1(\xh)^\top H_{uxx}\lam f_1(\xh)\dtt
=o(\gamma_k),
\end{split}
\ee
by Lemmas \ref{chap1lemmabound2} and \ref{chap1etaPont}.
The last inequality follows from integrating by parts one more time as it was done in \eqref{chap1Huxf1prime}.
Consider expression \eqref{chap1Rk}. By inequality
\eqref{chap1DeltaOmega'} and equation \eqref{chap1Huxx}, equality \eqref{chap1rest} is obtained and thus, the desired result follows.

\end{proof}


\begin{proof}

[of Theorem \ref{chap1sc2scalar}] \textit{Part 1.} First we prove that if $\wh$ is a normal Pontryagin minimum satisfying the $\gamma-$quadratic growth condition in the Pontryagin sense then \eqref{chap1unifpos} holds for some $\rho>0.$
Here the necessary condition of Theorem \ref{chap1NC2} is used.
Define $\yh(t):=\int_0^t \uh(s)\mr{d}s,$ and note that  $(\wh,\yh)$ is, for some $\rho'>0,$ a Pontryagin minimum of
\be\label{chap1PA}
\begin{split}
 &\tilde J:=J-\rho' \gamma(y-\yh,y(T)-\yh(T))\rightarrow \min,\\
 &\text{\eqref{chap1stateeq}-\eqref{chap1finalcons},}\quad
 \dot y=u,\quad y(0)=0.
\end{split}
\ee
Observe that the critical cone $\tilde\C_2$ for \eqref{chap1PA} consists of the points $(z,v,\delta y)$ in $\X_2 \times \U_2\times W^1_2(0,T;\cR)$ verifying $(z,v)\in\C_2,$ $\dot{\delta y}=v$ and $\delta y(0)=0.$ Since the pre-Hamiltonian at point $(\wh,\yh)$ coincides with the original pre-Hamiltonian, the set of multipliers for \eqref{chap1PA} consists of the points $(\lambda,\psi_y)$ with $\lambda\in\Lambda.$

Applying the second order necessary condition of Theorem \ref{chap1NC2}  at the point $(\wh,\yh)$ we see that,
for every $(z,v)\in\C_2$ and $\delta y(t):=\int_0^tv(s)\mr{d}s,$ there exists $\lambda\in\Lambda$ such that
\be\label{chap1Omegaminusgamma}
\Omega\lam(z,v)-\alpha_0\rho'(\|\delta y\|_2^2+\delta y^2(T))
\geq0,
\ee
where $\alpha_0\gr 0$ since $\wh$ is normal. Take $\rho:=\min_{\lambda\in\Lambda}\alpha_0\rho'\gr 0.$
Applying the Goh transformation in \eqref{chap1Omegaminusgamma}, condition \eqref{chap1unifpos} for the constant $\rho$ follows.

\textit{Part 2.} We shall prove that if \eqref{chap1unifpos} holds for some $\rho>0,$ then $\wh$ satisfies $\gamma-$quadratic growth in the Pontryagin sense.
On the contrary, assume that the
quadratic growth condition \eqref{chap1qgdef} is not valid. Consequently,  there exists
a sequence $\{v_k\}\subset \Uspace$ converging to 0 in the Pontryagin sense such that, denoting $u_k:=\uh+v_k,$
\begin{equation}\label{chap1qgrowth}
J(\uh+v_k)\leq J(\uh)+o(\gamma_k),
\end{equation}
where $y_k(t):=\int_0^tv_k(s)ds$ and
$\gamma_k:=\gamma(y_k,y_k(T)).$
 Denote by $x_k$ the solution
of equation \eqref{chap1stateeq} corresponding to $u_k,$ define
$w_k:=(x_k,u_k)$ and let $z_k$ be the
solution of \eqref{chap1lineareq} associated with $v_k.$   Take any $\lambda\in \Lambda.$ Multiply
inequality \eqref{chap1qgrowth} by $\alpha_0,$ add the
nonpositive term
$\sum_{i=0}^{d_{\varphi}}\alpha_i\varphi_i(x_k(T))+\sum_{j=1}^{d_{\eta}}\beta_j\eta_j(x_k(T))$
to its left-hand side, and obtain the inequality
\begin{equation}\label{chap1quadlag}
\Phi[\lambda](x_k,u_k)\leq\Phi[\lambda](\xh,\uh)+o(\gamma_k).
\end{equation}
Recall expansion \eqref{chap1taylor0}.
Let
$(\bar{y}_k,\hb_k):=(y_k,y_k(T))/\sqrt{\gamma_k}.$ Note that the
elements of this sequence have unit norm in $\Uspace_2\times
\mathbb{R}.$ By the Banach-Alaoglu Theorem, extracting if
necessary a sequence, we may assume that there exists
$(\bar{y},\hb)\in \Uspace_2\times \cR$  such that
\begin{equation}\label{chap1limityk}
\bar{y}_k\rightharpoonup \bar{y},\ \mr{and}\
\bar{h}_k\rightarrow \bar{h},
\end{equation}
where the first limit is taken in the weak topology of
$\Uspace_2.$
 The remainder of the proof is split into two parts.
\begin{itemize}
 \item [{\bf(a)}] Using equations
 \eqref{chap1taylor0} and \eqref{chap1quadlag} we prove that
$(\bar{\xi},\yb,\hb)\in\P_2,$ where $\bar{\xi}$ is a solution of \eqref{chap1xieq}.
 \item [{\bf(b)}] We prove that $(\yb,\hb)=0$ and that it is
the limit of $\{(\yb_k,\hb_k)\}$ in the strong sense. This
leads to a contradiction since each  $(\yb_k,\hb_k)$ has unit norm.
\end{itemize}

\noindent\textbf{(a)} We shall prove that $(\bar\xi,\bar{y},\hb)\in\P_2.$
From \eqref{chap1taylor0} and \eqref{chap1quadlag} it
follows that
\benl
0\leq \int_0^TH_u[\lambda](t)v_k(t)\dtt\leq
-\Omega_{\P_2}\lam(\xi_k,y_k,h_k)+o(\gamma_k),
\eenl
where $\xi_k$ is solution of \eqref{chap1xieq} corresponding to $y_k.$ The first inequality holds as $H_u[\lambda]v_k\geq 0$ almost everywhere on $[0,T]$ and we replaced $\Omega_{\P}$ by $\Omega_{\P_2}$ in view of Remark \ref{chap1OmegasRemark}.
By the continuity of mapping $\Omega_{\P_2}\lam$ over
$\X_2\times\Uspace_2\times \cR$ deduce that
\benl
0\leq \int_0^TH_u[\lambda](t)v_k(t)\dtt\leq O(\gamma_k),
\eenl
and thus, for each composing interval $(c,d)$ of $I_0,$
\be \label{chap1cseq1}
\lim_{k\rightarrow \infty} \int_c^d H_u
[\lambda](t)\varphi(t)\frac{v_k(t)}{\sqrt{\gamma_k}}dt=0,
\ee
for every nonnegative Lipschitz continuous function $\varphi$ with $\supp \varphi\subset (c,d).$
The latter expression means that the support of $\varphi$ is included in $(c,d).$
 Integrating by parts in \eqref{chap1cseq1} and by \eqref{chap1limityk} we
obtain
\benl\label{chap1cseq2}
0=\lim_{k\rightarrow \infty} \int_c^d
\ddt\left(H_u[\lambda](t)\varphi(t)\right)
\yb_k(t)\dtt=\int_c^d
\ddt\left(H_u[\lambda](t)\varphi(t)\right)
\yb(t)\dtt.
\eenl
By Lemma \ref{chap1l1}, $\yb$ is nondecreasing over $(c,d).$ Hence, in view of Lemma \ref{chap1l3}, we can integrate by parts in the previous equation to get
\be\label{chap1cseq3}
\int_c^dH_u\lam(t)\varphi(t)\mr{d}\yb(t)=0.
\ee
Take $t_0\in(c,d).$ By the strict complementary in Assumption 1, there exists $\lambda_0\in \Lambda$ such that $H_u[\lambda_0](t_0)>0.$ Hence, in view of the continuity of $H_u[\lambda_0]$, there exists $\varepsilon>0$ such that $H_u[\lambda_0]>0$  on $(t_0-2\varepsilon,t_0+2\varepsilon)\subset (c,d).$ Choose $\varphi$ such that $\supp\varphi\subset (t_0-2\varepsilon,t_0+2\varepsilon),$ and $H_u[\lambda_0](t)\varphi(t)=1$ on $(t_0-\varepsilon,t_0+\varepsilon).$
Since $\mr{d}\yb\geq0,$ equation \eqref{chap1cseq3} yields
\benl
\label{chap1cseq4}
\begin{split}
0
&=\int_c^d H_u[\lambda](t)\varphi(t) \mr{d}\yb(t)
\geq
\int_{t_0-\varepsilon}^{t_0+\varepsilon}
H_u[\lambda](t)\varphi(t) \mr{d}\yb(t)\\
&=
\int_{t_0-\varepsilon}^{t_0+\varepsilon} \mr{d}\yb(t)
=\yb(t_0+\varepsilon)-\yb(t_0-\varepsilon).
\end{split}
\eenl
As $\varepsilon$ and $t_0\in (c,d)$ are arbitrary we find that
\be\label{chap1dybar0}
\mr{d}\yb(t)=0,\quad \mr{on}\ I_0,
\ee
and thus \eqref{chap1yconstant} holds.
Let us prove condition \eqref{chap1yi0} for $(\bar{\xi},\yb,\hb).$ Suppose that $0\in I_0.$ Take $\varepsilon>0,$ and notice that by Assumption 1 there exists $\lambda'\in\Lambda$ and $\delta>0$ such that $H_u[\lambda'](t)>\delta$ for $t\in [0,d_1-\varepsilon],$ and thus by \eqref{chap1cseq1} we obtain $
\int_0^{d_1-\varepsilon}v_k(t)/\sqrt{\gamma_k}\dtt\rightarrow 0,$ as $v_k\geq 0.$ Then for all $s\in [0,d_1),$ we have
\benl
\yb_k(s)\rightarrow 0,
\eenl
and thus
\be\label{chap1yi0sc2}
\yb=0,\ \mr{on}\ [0,d_1),\ \mr{if}\ 0\in I_0.
\ee
Suppose that $T\in I_0.$ Then, we can derive
$\int_{a_N+\varepsilon}^T\vb_k(t)\dtt\rightarrow 0$  by an analogous argument. Thus, the pointwise convergence
\benl
\hb_k-\yb_k(s)\rightarrow 0,
\eenl
holds for every $s\in (a_N,T],$
and then,
\be\label{chap1yihsc2}
\yb=\hb,\ \mr{on}\ (a_N,T],\ \mr{if}\ T\in I_0.
\ee
It remains to check the final conditions \eqref{chap1translinearcons} for $\hb.$ Let $0\leq i\leq d_{\varphi},$
\be
\label{chap1phineg'}
\begin{split}
 \varphi'_i(\xh(T))(\bar{\xi}(T)+B(T)\hb)
&= \lim_{k\rightarrow \infty}\varphi'_i(\xh(T))\left(
\frac{\xi_k(T)+B(T)h_k}{\sqrt{\gamma_k}}\right)\\
&= \lim_{k\rightarrow
\infty}\varphi'_i(\xh(T))\frac{z_k(T)}{\sqrt{\gamma_k}}.
\end{split}
\ee
A first order Taylor expansion of the function $\varphi_i$
around $\xh(T)$ gives
\benl
\varphi_i(x_{k}(T))=\varphi_i(\xh(T))+\varphi'_i(\xh(T))\delta
x_k(T)+O(|\delta x_k(T)|^2).
\eenl
By Lemmas  \ref{chap1lemmabound2} and \ref{chap1etaPont} in the
Appendix, we can write
\benl
\label{chap1taylor_phi}
\varphi_i(x_{k}(T))=\varphi_i(\xh(T))+\varphi'_i(\xh(T))z_k(T)+o(\sqrt{\gamma_k}).
\eenl
Thus
\be
\label{chap1difphi}
\varphi'_i(\xh(T))\frac{z_k(T)}{\sqrt{\gamma_k}}=\frac{\varphi_i(x_{k}(T))-\varphi_i(\xh(T))}{\sqrt{\gamma_k}}+o(1).
\ee
Since $x_k$ satisfies \eqref{chap1finalcons}, equations \eqref{chap1phineg'} and \eqref{chap1difphi} yield, for $1\leq i\leq d_{\varphi}:$
\\
$\varphi'_i(\xh(T))(\bar{\xi}(T)+B(T)\hb)\leq 0.$ For $i=0$ use inequality \eqref{chap1qgrowth}.
Analogously,
\benl\label{chap1eta0}
\eta'_j(\xh(T))(\bar{\xi}(T)+B(T)\hb)= 0,\quad \mr{for}\
j=1,\hdots,d_{\eta}.
\eenl
Thus $(\bar{\xi},\yb,\hb)$ satisfies \eqref{chap1translinearcons}, and by \eqref{chap1dybar0}, \eqref{chap1yi0sc2} and  \eqref{chap1yihsc2},
we obtain
\benl
(\bar\xi,\yb,\hb)\in \P_2.
\eenl


\noindent\textbf{(b)} Return to the expansion
\eqref{chap1taylor0}.
Equation \eqref{chap1quadlag} and $H_u\lam\geq 0$ imply
\begin{equation*}
\begin{split}
\Omega_{\P_2}[\lambda]&(\xi_k,y_k,y_k(T))=\\
&\Phi[\lambda](x_k,u_k)-\Phi[\lambda](\xh,\uh)-\intT
H_u\lam v_k\dtt-o(\gamma_k)
\leq o(\gamma_k).
\end{split}
\end{equation*}
Thus
\begin{equation}\label{chap1lims}
\liminf_{k\rightarrow \infty}
\Omega_{\P_2}[\lambda](\bar\xi_k,\bar{y}_k,\bar{h}_k) \leq
\limsup_{k\rightarrow \infty}
\Omega_{\P_2}[\lambda](\bar\xi_k,\bar{y}_k,\bar{h}_k)\leq 0.
\end{equation}
Split $\Omega_{\P_2}$ as follows,
\benl
\Omega_{\P_2,w}\lam  (\xi,y,h):=\int_0^T\{ (Q\lam\xi,\xi)+ (M\lam\xi,y)\}\dtt+g\lam(\xi(T),h),
\eenl
\benl
\Omega_{\P_2,0}\lam (y):=\int_{I_0} (R\lam y,y)\dtt,
\eenl
and
\benl
\Omega_{\P_2,+}\lam(y):=\int_{I_+}
(R\lam y,y)\dtt.
\eenl
Notice that $\Omega_{\P_2,w}\lam$ is weakly continuous in the
space $\X_2\times\Uspace_2\times \cR.$ Consider now the subspace
\benl
\label{chap1chi2}
\Gamma_2:= \left\{
(\xi,y,h)\in \X_2\times\Uspace_2\times \cR:\eqref{chap1xieq},\ \eqref{chap1yconstant}\,\mr{and}\,\eqref{chap1yi0}\ \mr{hold}
\right\}.
\eenl
Notice that $\Gamma_2$ is itself a Hilbert space.
Let $\rho>0$ be the constant in the positivity condition \eqref{chap1unifpos} and define
\benl
\Lambda^{\rho}:=\{\lambda \in
\mr{co}\,\Lambda:\Omega_{\P_2}[\lambda]-\rho \gamma
\ \mr{is}\ \mr{weakly}\ \mr{l.s.c.}\ \mr{on}\ \Gamma_2\}.
\eenl
Equation \eqref{chap1unifpos} and Lemma
\ref{chap1quadform} in the Appendix imply that
\be
\label{chap1max}
\max_{\lambda\in\Lambda^{\rho}}\Omega_{\P_2}[\lambda](\bar\xi,\yb,\hb)\geq
\rho \gamma(\yb,\hb).
\ee
Denote by $\bar{\lambda}$ the element in $\Lambda^{\rho}$
that reaches the maximum in \eqref{chap1max}. Next we show that $R[\bar{\lambda}](t) \geq \rho$ on $I_+.$

Observe that $\Omega_{\P_2,0}[\bar{\lambda}]-\rho\int_{I_0}|y(t)|^2\dtt$ is weakly continuous in the space $\Gamma_2.$ In fact, consider a sequence $\{(\tilde\xi_k,\tilde y_k,\tilde h_k)\}\subset \Gamma_2$ converging weakly to some $(\tilde\xi,\tilde{y},\tilde{h})\in\Gamma_2.$ Since $\tilde{y}_k$ and $\tilde{y}$ are constant on $I_0,$ necessarily $\tilde{y}_k\rightarrow \tilde{y}$ uniformly in every compact subset of $I_0.$ Easily follows that
\be
\lim_{k\rightarrow\infty}
\Omega_{\P_2,0}[\bar{\lambda}](\tilde y_k) -\rho\int_{I_0}|\tilde y_k(t)|^2\dtt
= \Omega_{\P_2,0}[\bar{\lambda}](\tilde{y})-\rho\int_{I_0}|\tilde y(t)|^2\dtt,
\ee
and therefore, the weak continuity of $\Omega_{\P_2,0}[\bar{\lambda}]-\rho\int_{I_0}|y(t)|^2\dtt$ in $\Gamma_2$ holds.
Since $\Omega_{\P_2}[\bar{\lambda}]-\rho\gamma$ is weakly l.s.c. in $\Gamma_2,$ we get that the (remainder) quadratic mapping
\be
\label{chap1lscmapping}
y\mapsto \Omega_{\P_2,+}[\bar{\lambda}](y)-\rho\int_{I_+}|y(t)|^2\dtt,
\ee
is weakly l.s.c. on $\Gamma_2.$ In particular, it is weakly l.s.c. in the subspace of $\Gamma_2$ consisting of the elements for which $y=0$ on $I_0.$ Hence, in view of Lemma \ref{chap1legendre} in the Appendix, we get
\be\label{Rpos}
R[\bar{\lambda}](t)\geq \rho,\quad \mr{on}\ I_+.
\ee

The following step is proving the strong convergence of $\yb_k$ to $\yb.$ With this aim we make use of the uniform convergence on compact subsets of $I_0,$ which is pointed out in Lemma \ref{chap1unifconv}.

  Recall now Assumption 2, and let $N$ be the number of connected components of $I_0.$ Set $\varepsilon>0,$ and for each composing
interval $(c,d)$ of $I_0$, consider a smaller interval of
the form $(c+\varepsilon/2N,d-\varepsilon/2N).$ Denote
 their union as $I_0^{\varepsilon}$. Notice that
$I_0\backslash I_0^{\varepsilon}$ is of measure
$\varepsilon.$ Put $I_+^{\varepsilon}:=[0,T]\backslash
I_0^{\varepsilon}.$
By the Lemma \ref{chap1Rpos} in the Appendix, $R[\bar{\lambda}](t)$ is a continuous function of time, and thus from \eqref{Rpos} we can assure that $R[\bar{\lambda}](t) \geq \rho/2$ on
$I_+^{\varepsilon}$ for $\varepsilon$ sufficiently small.
Consequently,
\benl
\Omega_{\P_2,+}^{\varepsilon}[\bar{\lambda}](y):=\int_{I_+^{\varepsilon}}
(R[\bar{\lambda}] y,y)\dtt,
\eenl
is a Legendre form on $L_2(I_+^{\varepsilon}),$ and thus the following inequality holds for the approximating directions $\yb_k,$
\be
\label{chap1Omega+}
\Omega_{\P_2,+}^{\varepsilon}[\bar{\lambda}](\yb)\leq
\liminf_{k\rightarrow \infty}
\Omega_{\P_2,+}^{\varepsilon}[\bar{\lambda}](\yb_k).
\ee
Since the sequence $\yb_k$ converges uniformly to $\yb$ on every compact subset of $I_0,$ defining
\benl
\Omega_{\P_2,0}^{\varepsilon}[\bar{\lambda}](y):=\int_{I_0^{\varepsilon}} (R[\bar{\lambda}] y,y)\dtt,
\eenl
we get
\be
\label{chap1Omega0}
\lim_{k\rightarrow\infty}
\Omega_{\P_2,0}^{\varepsilon}[\bar{\lambda}](\bar\xi_k,\yb_k,\hb_k) = \Omega_{\P_2,0}^{\varepsilon}[\bar{\lambda}](\bar\xi,\bar{y},\bar{h}).
\ee
Notice that the weak continuity of $\Omega_{\P_2,0}^{\varepsilon}[\bar{\lambda}]$ in $\Gamma_2$ cannot be applied since
\\ $(\bar\xi_k,\yb_k,\hb_k)\notin \Gamma_2.$
From positivity condition \eqref{chap1unifpos}, equations \eqref{chap1Omega+}, \eqref{chap1Omega0}, and the weak continuity of $\Omega_{\P_2,w}[\bar{\lambda}]$ (in $\X_2\times \U_2 \times \cR)$ we get
\benl
\begin{split}
\rho \gamma (\bar{y},\hb)
\leq&\,\Omega_{\P_2}[\bar{\lambda}](\bar\xi,\bar{y},\bar{h})
\leq
\lim_{k\rightarrow\infty} \Omega_{\P_2,w}[\bar{\lambda}](\bar\xi_k,\yb_k,\hb_k)
+
\lim_{k\rightarrow\infty}
\Omega_{\P_2,0}^{\varepsilon}[\bar{\lambda}](\yb_k)\\
&+
\liminf_{k\rightarrow\infty}
\Omega_{\P_2,+}^{\varepsilon}[\bar{\lambda}](\yb_k)
=\liminf_{k\rightarrow\infty}
\Omega_{\P_2}[\bar\lambda](\bar{\xi}_k,\yb_k,\hb_k).
\end{split}
\eenl
On the other hand, inequality
 \eqref{chap1lims} implies that
the right-hand side of the last expression is nonpositive.
Therefore,
\benl
(\yb,\hb)=0,\ \mr{and}\ \lim_{k\rightarrow
\infty}\Omega_{\P_2}[\bar\lambda](\bar\xi_k,\yb_k,\hb_k)=0.
\eenl
Equation \eqref{chap1Omega0} yields $\lim_{k\rightarrow\infty}
\Omega_{\P_2,0}^{\varepsilon}[\bar{\lambda}](\bar\xi_k,\yb_k,\hb_k)=0$ and
thus
\be\label{chap1limOmega+}
\lim_{k\rightarrow
\infty}\Omega_{\P_2,+}^{\varepsilon}[\bar{\lambda}](\yb_k)=0.
\ee
We have: $\Omega_{\P_2,+}^{\varepsilon}[\bar{\lambda}]$ is a
Legendre form on $L_2(I_+^\varepsilon)$ and $\yb_k\rightharpoonup 0$ on
$I_+^{\varepsilon}.$ Thus, by \eqref{chap1limOmega+},
\benl
\yb_k\rightarrow 0,\quad \mr{on}\ L_2(I_+^\varepsilon).
\eenl
As we already noticed, $\{\yb_k\}$ converges uniformly on
$I^{\varepsilon}_0,$ thus the strong convergence holds
on $[0,T].$ Therefore
\be
(\yb_k,\hb_k)\longrightarrow (0,0),\quad \mr{on}\
\Uspace_2\times \cR.
\ee
This leads to a contradiction since $(\yb_k,\hb_k)$ has unit norm
for every $k\in \cN.$ Thus, $\wh$ is a Pontryagin minimum satisfying quadratic growth.

\end{proof}


\section{Extensions and an example}

\subsection{Including parameters}

Consider the following optimal control problem where the initial state is not determined, some parameters are included and a more general control constraint is considered.
\begin{align}
 &\label{chap1cost2} J:=\varphi_0(x(0),x(T),r(0))\rightarrow \min,\\
 &\label{chap1stateeq2}\dot{x}(t)=\sum_{i=0}^m u_i(t) f_i(x(t),r(t)),\\
 &\label{chap1rdot0}\dot r(t)=0,\\
 &\label{chap1controlcons2} a_i\leq u_i(t)\leq b_i,\ \mr{for}\ \mr{a.a.}\
t\in (0,T),\ i=1,\hdots,m \\
 &\label{chap1phicons2}  \varphi_i(x(0),x(T),r(0))\leq 0,\ \mathrm{for}\
i=1,\hdots,d_{\varphi},\\
&\label{chap1etacons2}\eta_j(x(0),x(T),r(0))=0,\ \mathrm{for}\
j=1\hdots,d_{\eta},
\end{align}
where $u\in\U,$ $x\in \X,$ $r\in \cR^{n_r}$ is a parameter considered as a state variable with zero-dynamics, $a,b\in \cR^m,$ functions
$f_i:\cR^{n+n_r}\rightarrow \cR^n,$ $\varphi_i:\cR^{2n+n_r}\rightarrow \cR,$  and
$\eta:\cR^{2n+n_r}\rightarrow \cR^{d_{\eta}}$ are twice continuously differentiable.
As  $r$ has zero dynamics, the costate variable $\psi_r$ corresponding to equation \eqref{chap1rdot0} does not appear in the pre-Hamiltonian. Denote with $\psi$ the costate variable associated with \eqref{chap1stateeq2}. The pre-Hamiltonian function for problem \eqref{chap1cost2}-\eqref{chap1etacons2} is given by
\benl
H\lam(x,r,u,t)=\psi(t)\sum_{i=0}^m u_i f_i(x,r).
\eenl
Let $(\xh,\rh,\uh)$ be a feasible solution for \eqref{chap1stateeq2}-\eqref{chap1etacons2}. Since $\rh(\cdot)$ is constant, we can denote it by $\rh.$ Assume that
\benl
\varphi_i(\xh(0),\xh(T),\rh)=0,\quad \mr{for}\ i=0,\hdots,d_{\varphi}.
\eenl
An element $\lambda=(\alpha,\beta,\psi_x,\psi_r)\in \cR^{d_{\varphi}+ d_{\eta}+1}\times W_{\infty}^1(0,T;\cR^{n,*})\times W_{\infty}^1(0,T;\cR^{n_r,*})$ is a Pontryagin multiplier for $(\xh,\rh,\uh)$ if it satisfies \eqref{chap1nontriv}, \eqref{chap1alphapos}, the costate equation for $\psi$
\benl\label{chap1costateeq2}
\left\{
\begin{split}
-\dot\psi_x(t)&=H_x\lam(\xh(t),\rh,\uh(t),t),\ \mr{a.e.}\ \mr{on}\ [0,T]\\
\psi_x(0)&=-\ell_{x_0}\lam(\xh(0),\xh(T),\rh),\\ \psi_x(T)&=\ell_{x_T}\lam(\xh(0),\xh(T),\rh),
\end{split}
\right.
\eenl
and for $\psi_r$
\be\label{chap1costateeqr}
\left\{
\begin{split}
-\dot\psi_r(t)&=H_r\lam(\xh(t),\rh,\uh(t),t),\ \mr{a.e.}\ \mr{on}\ [0,T]\\
\psi_r(0)&=-\ell_{r}\lam(\xh(0),\xh(T),\rh),\,\,\psi_r(T)=0.
\end{split}
\right.
\ee
Observe that \eqref{chap1costateeqr} implies the stationarity condition
\benl
\ell_r(\xh(0),\xh(T),\rh)+\intT H_r\lam(t)\dtt=0.
\eenl
\if{
\be\label{chap1termLag2}
\ell\lam(x_0,x,r):=\sum_{i=0}^{d_{\varphi}}\alpha_i\varphi_i(x(0),x(T))+\sum_{j=1}^{d_{\eta}}\beta_j\eta_j(x(0),x(T)).
\ee
}\fi
Take $v\in \U$ and consider the linearized state equation
\be
\label{chap1lineareq2}
\left\{
\begin{split}
 \dot z(t)&=\sum_{i=0}^m \uh_i(t)[f_{i,x}(\xh(t),\rh)z(t)+f_{i,r}(\xh(t),\rh)\delta r(t)]+\sum_{i=1}^m v_i(t)f_i(\xh(t),\rh),\\
 \dot{\delta r}(t)&=0,
\end{split}
\right.
\ee
where we can see that $\delta r(\cdot)$ is constant and thus we denote it by $\delta r.$ Let the linearized initial-final constraints be
\be
\label{chap1linearcons2}
\begin{split}
 &\varphi_i'(\xh(0),\xh(T),\rh)(z(0),z(T),\delta r)\leq 0,\quad\mr{for}\ i=1,\hdots,d_{\varphi},\\
&\eta_j'(\xh(0),\xh(T),\rh)(z(0),z(T),\delta r)=0,\quad\mr{for}\  j=1,\hdots,d_{\eta}.
\end{split}
\ee
Define for each $i=1,\hdots,m$ the sets
\begin{align*}
I^i_a&:=\{t\in [0,T]: \max_{\lambda\in \Lambda}H_{u_i}\lam (t)>0\},\\
I^i_b&:=\{t\in [0,T]: \max_{\lambda\in \Lambda}H_{u_i}\lam (t)<0\},\\
I^i_{\mr{sing}}&:=[0,T]\backslash (I^i_a\cup I^i_b).
\end{align*}

\noindent\textbf{Assumption 3.}
Consider the natural extension of Assumption 2, i.e.  for each $i=1,\hdots,m,$ the sets $I^i_a$ and $I^i_b$ are  finite unions of
intervals, i.e.
\benl
I^i_a=\ds\bigcup_{j=1}^{N_a^i} I^i_{j,a},\quad
I^i_b=\ds\bigcup_{j=1}^{N_b^i} I^i_{j,b},
\eenl
for $I^i_{j,a}$ and $I^i_{j,b}$ being subintervals of $[0,T]$ of the form $[0,c),$
$(d,T];$ or $(c,d)$ if $c\neq 0$ and $d\neq T.$
Notice that $I^i_a\cap I^i_b=\emptyset.$
Call
$c_{1,a}^i<d_{1,a}^i<c_{2,a}^i<\hdots<c_{N_a^i,a}^i<d_{N_a^i,a}^i$ the
endpoints of these intervals corresponding to bound $a,$ and  define them analogously for $b.$ Consequently,
$I_{\mr{sing}}^i$ is a finite union of intervals as well.
Assume that a concatenation of a bang arc followed by another bang arc is forbidden.

\noindent\textbf{Assumption 4.}
Strict complementarity assumption for control constraints:
\benl
\left\{
\begin{split}
&I^i_a=\{t\in [0,T]:\uh_i(t)=a_i\},\ \mr{up}\ \mr{to}\ \mr{a}\ \mr{set}\ \mr{of}\ \mr{null}\ \mr{measure,}\\
&I^i_b=\{t\in [0,T]:\uh_i(t)=b_i\},\ \mr{up}\ \mr{to}\ \mr{a}\ \mr{set}\ \mr{of}\ \mr{null}\ \mr{measure.}
\end{split}
\right.
\eenl

Consider
\benl
\C_2:=
\left\{
\begin{split}
&(z,\delta r,v)\in  \X_2\times\cR^{n_r}\times\Uspace_2:\text{\eqref{chap1lineareq2}-\eqref{chap1linearcons2}}\ \mr{hold},\\
&v_i=0\ \mathrm{on}\ I^i_{\mr{a}}\cup I^i_{\mr{b}},\ \mr{for}\ i=1,\hdots,m
\end{split}
\right\}.
\eenl
The Goh transformation allows us to obtain variables $(\xi,y)$ defined by
\benl
 y(t):=\int_0^t v(s)\mr{d}s,\ \xi:=z-\sum_{i=1}^my_if_i.
\eenl
Notice that $\xi$ satisfies the equation
\be \label{chap1xieq2}
\begin{split}
\dot\xi&=A^x\xi+A^r\delta r+B_1^xy
,\\
\xi(0)&=z(0),
\end{split}
\ee
where, denoting $[f_i,f_j]^x:=f_{i,x}f_j-f_{j,x}f_i,$
\benl
A^x:=\sum_{i=0}^m \uh_if_{i,x},\quad
A^r:=\sum_{i=0}^m \uh_if_{i,r},\quad
B_1^xy:=\sum_{j=1}^m y_j\sum_{i=0}^m\uh_i[f_i,f_j]^x.
\eenl
Consider the transformed version of \eqref{chap1linearcons2},
\be\label{chap1transcons2}
\begin{split}
 &\varphi_i'(\xh(0),\xh(T),\rh)(\xi(0),\xi(T)+B(T)h,\delta r)\leq 0,\
i=1,\hdots,d_{\varphi},\\
&\eta_j'(\xh(0),\xh(T),\rh)(\xi(0),\xi(T)+B(T)h,\delta r)=0,\ j=1,\hdots,d_{\eta},
\end{split}
\ee
and let the cone $\P$ be given by
\benl
\mathcal{P}:=
\left\{
\begin{split}
&(\xi,\delta r,y,h)\in \X\times\cR^{n_r}\times\Y\times \cR^m:
y(0)=0,\ h=y(T),\\
&\eqref{chap1xieq2}\ \mr{and}\ \eqref{chap1transcons2}\ \mr{hold},
\ y_i'=0\ \mathrm{on}\
I^i_{\mr{a}}\cup I^i_{\mr{b}},\ \mr{for}\ i=1,\hdots,m
\end{split}
\right\}.
\eenl
Observe that each $(\xi,\delta r,y,h)\in \P$ satisfies
\be\label{chap1yconstant'}
y_i\ \mr{constant}\ \mr{over}\ \mr{each}\ \mr{composing}
\ \mr{interval}\ \mr{of}\ I^i_{\mr{a}}\cup I^i_{\mr{b}},
\ee
and at the endpoints,
\be\label{chap1yi0'}
\left\{
\ba{l}
y_i=0\ \mr{on}\
[0,d],\ \mr{if}\ 0\in I^i_{\mr{a}}\cup I^i_{\mr{b}},\ \mr{and},\\
y_i=h_i\ \mr{on}\
[c,T],\ \mr{if}\ T\in I^i_{\mr{a}}\cup I^i_{\mr{b}},
\ea
\right.
\ee
where $[0,d)$ is the first maximal composing interval of $I^i_{\mr{a}}\cup I^d_{\mr{b}}$ when $0\in I^d_{\mr{a}}\cup I^d_{\mr{b}},$ and $(c,T]$ is its last composing interval when $T\in I^i_{\mr{a}}\cup I^i_{\mr{b}}.$ Define
\benl
\P_2:= \left\{
\begin{split}
&(\xi,\delta r,y,h)\in \X_2\times\cR^{n_r}\times\U_2\times \cR^m:\\
&\eqref{chap1xieq2},\ \eqref{chap1transcons2},\, \eqref{chap1yconstant'}\ \mr{and}\ \eqref{chap1yi0'}\ \mr{hold}\ \mr{for}\ i=1,\hdots,m
\end{split}
\right\}.
\eenl
Recall definitions in equations \eqref{chap1QCM}, \eqref{chap1SV}, \eqref{chap1R}, \eqref{chap1g}, \eqref{chap1OmegaG}. Minor simplifications appear in the computations of these functions as the dynamics of $r$ are null and $\delta r$ is constant. We outline these calculations in an example.

Consider $M\subset \cR^s$ and the subset of $M\subset \cR^s$ defined by
\benl
G(M):=\{\lambda\in M:V_{ij}\lam=0\ \mr{on}\ I^i_{\mr{sing}}\cap
I^j_{\mr{sing}},\ \mr{for}\ \mr{every}\ \mr{pair}\ 1\mi i\neq j\leq m\}.
\eenl
Using the same techniques, we obtain the equivalent of Theorem \ref{chap1NCnew}:
\begin{corollary}
 \label{chap1NCnew2}
Suppose that $(\xh,\rh,\uh)$ is a weak minimum for problem \eqref{chap1cost2}-\eqref{chap1etacons2}. Then
\benl
\max_{\lambda\in G(\mr{co}\,\Lambda)} \Omega_{\P_2}\lam (\xi,\delta r,y,h)\geq 0,\quad \mr{for}\ \mr{all}\ (\xi,\delta r,y,h)\in \P_2.
\eenl
\end{corollary}

\if{
Consider the natural extension of conditions \eqref{chap1hypsc2}:
\be\label{chap1hypsc2'}
\begin{split}
\forall\lambda\in\Lambda:
 &\ H_u\lam(0)>0,\ \mr{if}\ 0\in I_a,\quad H_u\lam(T)>0,\ \mr{if}\ T\in I_a,\\
&\ H_u\lam(0)<0,\ \mr{if}\ 0\in I_b,\quad H_u\lam(T)<0,\ \mr{if}\ T\in I_b.
\end{split}
\ee
}\fi

By a simple adaptation of the proof of Theorem \ref{chap1sc2scalar} we get the equivalent result.
\begin{corollary}\label{chap1sc2scalar2}
 Let $m=1.$ Suppose that there exists $\rho>0$ such that
\be\label{chap1unifpos'}
\max_{\lambda\in\Lambda} \Omega_{\P_2}\lam(\xi,\delta r,y,h)\geq \rho\gamma(y,h),\quad \mr{for}\ \mr{all}\ (\xi,\delta r,y,h)\in\P_2.
\ee
Then $(\xh,\rh,\uh)$ is a
 Pontryagin minimum that satisfies $\gamma-$quadratic
growth.
\end{corollary}

\if{
 of matrices $Q,$ $M$ and $R$ in \eqref{chap1QCM} and \eqref{chap1R}, and of $\Omega_{\P}$ in \eqref{chap1OmegaG}. Simple computations lead us to
\be
Q\lam=
\begin{pmatrix}
 H_{xx}&H_{xr}\\
 H_{rx}&H_{rr}
\end{pmatrix},
\quad
M\lam=-\psi
\left(\frac{\partial}{\partial x}[f_1,f_0]^x,
\frac{\partial}{\partial r}[f_1,f_0]^x
\right)
\ee
\be
R\lam=\psi[f_1,[f_1,f_0]^x]^x,\quad
C\lam=\psi(f_{1,x},f_{1,r}),
\ee
\be
\begin{split}
g\lam (\zeta_0,\zeta_T,\rho, h)=&
\ \ell''\lam(\xh(0),\xh(T),\rh)(\zeta_0,\zeta_T+f_1(\xh(T))h,\rho)^2
\\
&+(C\lam(T)(2\zeta_T+f_1(\xh(T))h,h),
\end{split}
\ee
where
\be
[f_i,f_j]^x:=f_{i,x}f_j-f_{j,x}f_i.
\ee
We obtain
\be
\begin{split}
\Omega_{\P_2}\lam(\xi,\delta r,y,h)=&
g\lam(\xi(0),\xi(T),\delta r,h)\\
&+
\intT
\left[
Q\lam
\begin{pmatrix}
 \xi\\
\delta r
\end{pmatrix}
^2
+2yM\lam
\begin{pmatrix}
 \xi\\
\delta r
\end{pmatrix}
+y^2R\lam
\right]
\dtt.
\end{split}
\ee
From Theorem \ref{chap1sc2scalar} follows
\begin{corollary}
\label{chap1sc2scalar'}
Suppose that there exists $\alpha>0$ such that
\be
\max_{\lambda\in\Lambda} \Omega_{\P_2}\lam(\xi,\delta r,y,h)\geq \alpha\gamma(\xi,\delta r,y,h),\quad \forall (\xi,\delta r,y,h)\in \P_2.
\ee
Thus $\wh=(\xh,\rh,\uh)$ is a local
 Pontryagin minimum of \eqref{chap1cost2}-\eqref{chap1etacons2} that satisfies $\gamma-$quadratic
growth.
\end{corollary}
}\fi

\subsection{Application to minimum-time problems}

Consider the problem
\benl
\begin{split}
& J:=T\rightarrow \min,\\
& \mr{s.t.}\ \eqref{chap1stateeq2}-\eqref{chap1etacons2}.
\end{split}
\eenl
Observe that by the change of variables:
\be\label{chap1changevar}
x(s)\leftarrow x(Ts),\quad u(s)\leftarrow u(Ts),
\ee
we can transform the problem into the following formulation.
\begin{align*}
 &J:=T(0)\rightarrow \min,\\
 &\dot{x}(s)=T(s)\sum_{i=0}^m u_i(s) f_i(x(s),r(s)),\quad \mr{a.e.}\ \mr{on}\ [0,1],\\
 &\dot r(s)=0,\quad \mr{a.e.}\ \mr{on}\ [0,1],\\
 &\dot T(s)=0,\quad \mr{a.e.}\ \mr{on}\ [0,1],\\
 & a_i\leq u_i(s)\leq b_i,\quad \mr{a.e.}\ \mr{on}\ [0,1],\ i=1,\hdots,m, \\
 &  \varphi_i(x(0),x(1),r(0))\leq 0,\ \mathrm{for}\
i=1,\hdots,d_{\varphi},\\
&\eta_j(x(0),x(T),r(0))=0,\ \mathrm{for}\
j=1\hdots,d_{\eta}.
\end{align*}
We can apply Corollaries \ref{chap1NCnew2} and \ref{chap1sc2scalar2} to the problem written in this form.
We outline the calculations in the following example.

\subsubsection{Example: Markov-Dubins problem}
Consider a problem over the interval $[0,T]$ with free final time $T:$
\be\label{chap1markov}
\begin{split}
&J:=T\rightarrow \min,\\
&\dot{x}_1=-\sin x_3,\ x_1(0)=0,\ x_1(T)=b_1,\\
&\dot{x}_2=\cos x_3,\ x_2(0)=0,\ x_2(T)=b_2,\\
&\dot{x}_3=u,\ x_3(0)=0,\ x_3(T)=\theta,\\
&-1\leq u\leq 1,
\end{split}
\ee
with $0<\theta <\pi,$ $b_1$ and $b_2$ fixed.

This problem  was originally introduced by Markov in \cite{Mar87} and studied by Dubins in \cite{Dub57}. More recently, the problem was investigated by Sussmann and Tang \cite{SusTan91}, Soueres and Laumond \cite{SouLau96}, Boscain and Piccoli \cite{BosPic}, among others.

Here we will study the optimality of the extremal
\be
\uh(t):=
\left\{
\ba{cl}
1\ &\mr{on}\ [0,\theta],\\
0\ &\mr{on}\ (\theta,\Th].
\ea
\right.
\ee
Observe that by the change of variables \eqref{chap1changevar} we can transform \eqref{chap1markov} into the following problem on the interval $[0,1].$
\be\label{chap1markovT}
\begin{split}
&J:=T(0)\rightarrow \min,\\
&\dot{x}_1(s)=-T(s)\sin x_3(s),\ x_1(0)=0,\ x_1(1)=b_1,\\
&\dot{x}_2(s)=T(s)\cos x_3(s),\ x_2(0)=0,\ x_2(1)=b_2,\\
&\dot{x}_3(s)=T(s)u(s),\ x_3(0)=0,\ x_3(1)=\theta,\\
&\dot T(s)=0,\\
&-1\leq u(s)\leq 1.
\end{split}
\ee
\if{
We could then apply Corollary \ref{chap1sc2scalar2}
to prove the optimality of $(\xh,\Th,\uh).$ Actually, in this special case, it is not necessary to use the change of variables. Fixing final time to $\Th$ and changing cost function to  $J:=1,$ we can show that the second order sufficient condition for this problem holds, and thus candidate solution $(\xh,\Th,\uh)$ is a Pontryagin minimum. In fact, as cost functional is constant, $(\xh,\Th,\uh)$ is an isolated feasible solution.
}\fi
We obtain for state variables:
\be
\label{chap1x3}
\xh_3(s)=
\left\{
\ba{cl}
\Th s\ &\mr{on}\ [0,{\theta}/\Th],\\
\theta\ &\mr{on}\ ({\theta}/\Th,1],
\ea
\right.
\ee
\benl\label{chap1x1}
\xh_1(s)=
\left\{
\ba{cl}
\cos (\Th s)-1\ &\mr{on}\ [0,{\theta}/\Th],\\
\Th\sin{\theta}({\theta}/\Th-s)+\cos{\theta}-1 \ &\mr{on}\ ({\theta}/\Th,1],
\ea
\right.
\eenl
\benl
\xh_2(s)=
\left\{
\ba{cl}
\sin {\Th s}\ &\mr{on}\ [0,{\theta}/\Th],\\
\Th\cos{{\theta}}(s-{\theta}/\Th)+\sin{\theta}\ &\mr{on}\ (\theta,\Th].
\ea
\right.
\eenl
Since the terminal values for $x_1$ and $x_2$ are fixed,
the final time $\Th$ is determined by the previous equalities.
\if{
Consider the problem in the interval $[\theta,T],$  i.e,
\be
\begin{split}
&J:=T\leftarrow \min\\
&\dot{x}_1=-\sin x_3,\ x_1(\theta)=a+\cos\theta,\ x_1(T)=0\\
&\dot{x}_2=-\cos x_3,\ x_2(\theta)=\sin\theta,\ x_2(T)=0\\
&\dot{x}_3=u,\ x_3(\theta)=\theta,\ x_3(T)=\theta\\
&-1\leq u\leq 1
\end{split}
\ee
}\fi
The pre-Hamiltonian for problem \eqref{chap1markovT}
is
\be\label{chap1Hexam}
H\lam(s):=T(s)(-\psi_1(s)\sin x_3(s)+\psi_2(s)\cos  x_3(s)+\psi_3(s) u(s)).
\ee
The final Lagrangian is
\benl
\ell:=\alpha_0 T(1)+\sum_{j=1}^3 (\beta^j x_j(0)+
\beta_jx_j(1)).
\eenl
As $\dot{\psi}_1\equiv 0,$ and $\dot{\psi}_2\equiv 0,$ we get
\benl
\psi_1\equiv \beta_1,\quad \psi_2\equiv \beta_2,\quad \mr{on}\ [0,1].
\eenl
Since the candidate control $\uh$ is singular on $[{\theta}/\Th,1],$  we have $H_u\lam\equiv 0.$
By \eqref{chap1Hexam}, we obtain
\be\label{chap1psi30}
\psi_3(s)=0,\quad \mr{on}\ [{\theta}/\Th,1].
\ee
Thus $\beta_3=0.$ In addition, as the costate equation for $\psi_3$ is
\benl
-\dot{\psi}_3=\Th(-\beta_1\cos \hat x_3-\beta_2\sin \hat x_3),
\eenl
by \eqref{chap1x3} and \eqref{chap1psi30}, we get
\be\label{chap1exeq1}
\beta_1\cos {\theta}+\beta_2 \sin {\theta}=0.
\ee
From \eqref{chap1x3} and \eqref{chap1psi30} and since $H$ is constant and equal to $-\alpha_0,$
we get
\be\label{chap1Hexam2}
H=\Th(-\beta_1\sin {\theta}+\beta_2\cos {\theta})\equiv -\alpha_0.
\ee
\begin{proposition}\label{chap1exprop} The following properties hold
\begin{itemize}
\item[(i)] $\alpha_0>0,$
\item[(ii)] $H_u\lam(s)<0$ on $[0,{\theta}/\Th)$ for all $\lambda\in\Lambda.$
\end{itemize}
\end{proposition}

\begin{proof} \textbf{Item (i)}
Suppose that $\alpha_0=0.$ By \eqref{chap1exeq1} and \eqref{chap1Hexam2},
we obtain
\benl
\beta_1\cos\theta+\beta_2\sin\theta=0,\,\,\mr{and}\,\,
-\beta_1\sin\theta+\beta_2\cos\theta=0.
\eenl
Suppose, w.l.g., that $\cos\theta\neq 0.$
Then $\beta_1=-\beta_2\ds\frac{\sin\theta}{\cos\theta}$ and thus
\benl
\beta_2\ds\frac{\sin^2\theta}{\cos\theta}+\beta_2\cos\theta=0.
\eenl
We conclude that $\beta_2=0$ as well. This implies $(\alpha_0,\beta_1,\beta_2,\beta_3)=0,$ which contradicts the non-triviality condition \eqref{chap1nontriv}. So, $\alpha_0>0,$ as required.

\noindent\textbf{Item (ii)}
Observe that
\benl
H_{u}\lam(s)\leq 0,\quad \mr{on}\ [0,\theta/\Th),
\eenl
and $H_u\lam=\psi_3.$ Let us prove that $\psi_3$ is never 0 on $[0,\theta/\Th).$ Suppose there exists $s_1\in [0,\theta/\Th)$ such that $\psi_3(s_1)=0.$ Thus, since $\psi_3(\theta/\Th)=0$ as indicated in \eqref{chap1psi30}, there exists $s_2\in (s_1,\theta/\Th)$ such that $\dot\psi_3(s_2)=0,$ i.e.
\be\label{chap1exeq2}
\beta_1\cos (\Th s_2)+\beta_2\sin (\Th s_2)=0.
\ee
Equations \eqref{chap1exeq1} and \eqref{chap1exeq2} imply that $
\tan (\theta/\Th)=\tan (s_2/\Th).$
This contradicts $\theta<\pi.$ Thus $\psi_3(s)\neq 0$ for every $s\in [0,\theta/\Th),$ and consequently,
\benl
H_u\lam(s) <0,\quad \mr{for}\ s\in[0,\theta/\Th).
\eenl
\end{proof}

Since $\alpha_0\gr 0,$ then $\delta T=0$ for each element of the critical cone, where $\delta T$ is the linearized state variable $T.$
Observe that as $\uh=1$ on $[0,\theta/\Th],$ then
\benl
y=0\ \mr{and}\ \xi=0,\ \mr{on}\ [0,\theta/\Th],\ \mr{for}\ \mr{all}\ (\xi,\delta T,y,h)\in\P_2.
\eenl
We look for the second variation in the interval $[\theta/\Th,1].$ The Goh transformation gives
\benl
\xi_3=z_3-\Th y,
\eenl
and since $\dot{z}_3=\Th v,$ we get $z_3=\Th y$ and thus $\xi_3=0.$ Then, as $H_{ux}=0$ and $\ell''=0,$ we get
\benl
\Omega\lam=\int_{\theta/\Th}^1 (\beta_1\sin\theta-\beta_2\cos\theta)y^2dt=\alpha_0\int_{0}^1y^2dt.
\eenl
Notice that if $(\xi,\delta T,y,h)\in \P_2,$ then $h$ satisfies $\xi_3(T)+\Th h=0,$ and, as $\xi_3(T)=0,$ we get $h=0.$  Thus
\benl
\Omega\lam(\xi,y,h)=\alpha_0\int_0^Ty^2dt=\alpha_0\gamma(y,h),\quad \mr{on}\ \P_2.
\eenl
Since Assumptions 3 and 4 hold, we conclude by Corollary \ref{chap1sc2scalar2}
that $(\xh,\Th,\uh)$ is a Pontryagin minimum satisfying quadratic growth.

\section{Conclusion}

We provided a set of necessary and sufficient conditions for a bang-singular extremal. The sufficient condition is restricted to the scalar control case.
These necessary and sufficient conditions are close in the sense that, to pass from one to the other, one has to strengthen a non-negativity inequality transforming it into a coercivity condition.

This is the first time that a sufficient condition that is `almost necessary'  is established for a bang-singular extremal for the general Mayer problem. In some cases the condition can be easily checked as it can be seen in the example.

\section{Appendix}

\begin{lemma}\label{chap1PdenseHadd}
Let
\benl
X:=\{(\xi,y,h)\in \X_2\times\Uspace_2\times \cR^m:\ \eqref{chap1xieq},\text{\eqref{chap1yconstant}-\eqref{chap1yi0}}\
\mr{hold}\},
\eenl
\benl
L:=\{(\xi,y,y(T))\in \X\times\Y\times \cR^m: y(0)=0,
\ \eqref{chap1xieq}\ \mr{and}\ \eqref{chap1yconstant}\}.
\eenl
Then $L$ is a dense subset of $X$ in the $\X_2\times\Uspace_2\times \cR^m-$topology.
\end{lemma}

\begin{proof} (See Lemma 6 in \cite{DmiShi10}.) Let us prove the result for $m=1.$ The general case is a trivial extension.
Let  $(\bar\xi,\yb,\hb)\in X$ and $\varepsilon,\delta\gr 0.$
Consider $\phi\in\Y$ such that $\|\yb-\phi\|_2\mi \varepsilon/2.$
 In order to satisfy condition \eqref{chap1yi0} take
\benl
\left\{
\ba{ll}
y_{\delta}(t):= 0,\quad \mr{for}\ t\in[0,d_1],&\mr{if}\ c_1=0,\\
y_{\delta}(t):= h,\quad \mr{for}\ t\in[c_N,T],&\mr{if}\ d_N=T,
\ea
\right.
\eenl
where $c_j,d_j$ were introduced in Assumption 2.
Since $\yb$ is constant on each $I_j,$ define $y_{\delta}$ constant over these intervals with the same constant value as $\yb.$ It remains to define $y_{\delta}$ over $I_+.$
Over each maximal composing interval $(a,b)$ of $I_+,$ define $y_{\delta}$ as described below.
 Take $c:=\yb(a-)$ if $a\gr0,$ or $c:=0$ if $a=0;$ and let $d:=\yb(b+)$ if $b\mi T,$ or $d:=h$ when $b=T.$
Define two affine functions $\ell_{1,\delta}$ and $\ell_{2,\delta}$ satisfying
\be\label{chap1elldelta}
\begin{split}
&\ell_{1,\delta}(a)=c,\ \ell_{1,\delta}(a+\delta)=\phi(a+\delta),\\
&\ell_{2,\delta}(b)=d,\ \ell_{2,\delta}(b-\delta)=\phi(b-\delta).
\end{split}
\ee
Take
\be\label{chap1ydelta}
y_{\delta}(t):=
\left\{
\begin{split}
\ell_{1,\delta}(t),\quad &\mr{for}\ t\in [a,a+\delta],\\
\phi(t),\quad &\mr{for}\ t\in (a+\delta,b-\delta),\\
\ell_{2,\delta}(t),\quad &\mr{for}\ t\in [b-\delta,b],
\end{split}
\right.
\ee
and notice that
$
\|\phi-y_{\delta}\|_{2,[a,b]}\leq \frac{1}{k}\max{(|c|,|d|,M)},
$
where
$
M:=\sup_{t\in [a,b]}|\phi(t)|.
$
Finally, observe that $y_{\delta}(T)=h,$ and, for sufficiently small $\delta,$
\benl
\|\yb-y_{\delta}\|_2\leq \|\yb-\phi\|_2+\|\phi-y_{\delta}\|_2\mi \varepsilon.
\eenl
Thus, the result follows.
\end{proof}

\begin{lemma}
\label{chap1barell}
Let $\lambda\in\Lambda$ and $(z,v)\in\C_2.$ Then
\be\label{chap1barelleq}
\sum_{i=0}^{d_{\varphi}}\alpha_i\bar\varphi_i''(\uh)(v,v)+\sum_{i=1}^{d_{\eta}}
\beta_j\bar\eta_j''(\uh)(v,v)=\Omega\lam(z,v).
\ee
\end{lemma}

\begin{proof}
Let us compute the left-hand side of \eqref{chap1barelleq}. Notice that
\be\label{chap1secondell}
\sum_{i=0}^{d_{\varphi}}\alpha_i\bar\varphi_i(\uh)
+\sum_{i=1}^{d_{\eta}}\beta_j\bar\eta_j(\uh)=\ell\lam(\xh(T)).
\ee
Let us look for a second order expansion for $\ell.$
Consider first a second order expansion of the state variable:
\benl
x=\xh+z+\half z_{vv}+o(\|v\|_{\infty}^2),
\eenl
where $z_{vv}$ satisfies
\be\label{chap1zvv}
\dot z_{vv}=Az_{vv}+D^2_{(x,u)^2}F(\xh,\uh)(z,v)^2,\quad z_{vv}(0)=0,
\ee
with $F(x,u):=\sum_{i=0}^mu_if_i(x).$ Consider the second order expansion for $\ell:$
\be\label{chap1ellexp}
\begin{split}
\ell\lam(x(T))
=&\ell\lam((\xh+z+\half z_{vv})(T))+o(\|v\|_1^2)\\
=&\ell\lam(\xh(T))+\ell'\lam(\xh(T))(z(T)+\half z_{vv}(T))\\
&+\half\ell''\lam(\xh(T))(z(T)+\half z_{vv}(T))^2+o(\|v\|_1^2).
\end{split}
\ee
\textbf{Step 1.} Compute
\benl
\begin{split}
\ell'&\lam(\xh(T))z_{vv}(T)
=\psi(T)z_{vv}(T)-\psi(0)z_{vv}(0)\\
&=\int_0^T [\dot\psi z_{vv}+\psi\dot z_{vv}]\dtt
=\int_0^T
\{-\psi Az_{vv}+\psi(Az_{vv}+D^2F_{(x,u)^2}(z,v)^2)\}\dtt\\
&=\int_0^T D^2H\lam(z,v)^2\dtt.
\end{split}
\eenl
\textbf{Step 2.} Compute $\ell''\lam(\xh(T))(z(T),z_{vv}(T)).$
Applying Gronwall's Lemma, we obtain $
\|z\|_{\infty}=O(\|v\|_1),$ and $\|z_{vv}\|_{\infty}=O(\|v^2\|_1).$
Thus
\benl
|(z(T),z_{vv}(T))|=O(\|v\|^3_1),
\eenl
and we conclude that
\benl
|\ell''\lam(\xh(T))(z(T),z_{vv}(T))|=O(\|v\|^3_1).
\eenl
\textbf{Step 3.} See that $\ell''\lam(\xh(T))(z_{vv}(T))^2=O(\|v\|_1^4).$
Then  by \eqref{chap1ellexp} we get,
\benl
\begin{split}
\ell\lam(x(T))=&\ell\lam(\xh(T))+\ell'\lam(\xh(T))z(T)\\
&+\half\ell''\lam(\xh(T))z^2(T)+\half\int_0^T D_{(x,u)^2}^2H\lam(z,v)^2\dtt+o(\|v\|_1^2)\\
=&\ell\lam(\xh(T))+\ell'\lam(\xh(T))z(T)+\Omega\lam(z,v)+o(\|v\|_1^2).
\end{split}
\eenl
The conclusion follows by \eqref{chap1secondell}.
\end{proof}

\begin{lemma}\label{chap1lemmabound1}
Given $(z,v)\in\W$ satisfying \eqref{chap1lineareq}, the following estimation holds for some $\rho\gr 0:$
\benl
\|z\|_2^2+|z(T)|^2\leq \rho \gamma(y,y(T)),
\eenl
where $y$ is defined by \eqref{chap1yxi}.
\begin{remark} $\rho$ depends on $\wh,$ i.e.  it does not vary with $(z,v).$
\end{remark}

\end{lemma}

\begin{proof}
Every time we mention $\rho_i$ we are referring to a constant depending on $\|A\|_{\infty},$ $\|B\|_{\infty}$ or both.
Consider $\xi,$ the solution of equation \eqref{chap1xieq} corresponding to $y.$ Gronwall's Lemma and the Cauchy-Schwartz inequality imply
\be\label{chap1lemmazxit}
\|\xi\|_{\infty}\leq \rho_1 \|y\|_2.
\ee
This last inequality, together with expression \eqref{chap1yxi}, implies
\be
\label{chap1lemmazz}
\|z\|_2\leq \|\xi\|_2+\|B\|_{\infty}\|y\|_2\leq
\rho_2 \|y\|_2.
\ee
On the other hand, equations \eqref{chap1yxi} and
\eqref{chap1lemmazxit} lead to
\benl
|z(T)|\leq |\xi(T)|+\|B\|_{\infty}|y(T)|\leq
\rho_1 \|y\|_2+\|B\|_{\infty}|y(T)|.
\eenl
Then, by the inequality ${ab}\leq \frac{a^2+b^2}{2},$ we get
\be\label{chap1lemmazzT}
|z(T)|^2\leq
\rho_3(\|y\|_2^2+|y(T)|^2).
\ee
The conclusion follows from equations \eqref{chap1lemmazz}
and \eqref{chap1lemmazzT}.
\end{proof}

The next lemma is a generalization of  the previous result to the nonlinear case. See Lemma 6.1 in Dmitruk \cite{Dmi87}.

\begin{lemma}
\label{chap1lemmabound2}
Let $w=(x,u)$ be the solution of \eqref{chap1stateeq} with $\|u\|_2\leq c$ for some constant $c.$
Put $(\delta x,v):=w-\wh.$ Then
\benl
|\delta x(T)|^{2}+\|\delta x\|_2^{2}\leq \rho
\gamma(y,y(T)),
\eenl
where $y$ is defined by \eqref{chap1yxi} and $\rho$ depends on $c.$
\end{lemma}

\if{
\begin{proof}
Consider
\be\label{chap1l2xi}
\xi:=\delta x-B y.
\ee
Differentiating with respect to the time variable follows
\be\label{chap1l2xidot}
\begin{split}
 \dot{\xi}
=\sum_{i=0}^m u_i f_i(x)-B \uh
-\dot By-Bv.
\end{split}
\ee
Take $L$ the maximum of the Lipschitz constants of all $f_i,$ with $i=1,\hdots,m,$ and obtain from \eqref{chap1l2xidot}:
\benl
|\dot{\xi}(t)|
\leq
L|\delta x(t)|(1+|u(t)|)+\|\dot{B}\|_{\infty}|y(t)|.
\eenl
Replacing $\delta x$ by its expression in \eqref{chap1l2xi} obtain
\benl
|\dot{\xi}(t)|
\leq
L(1+|u(t)|)|\xi(t)|+(L\|B\|_{\infty}(1+|u(t)|)+\|\dot{B}\|_{\infty})|y(t)|
\eenl
Applying Gronwall's Lemma and Cauchy-Schwartz inequality to the last expression follows
\be \label{chap1l2xit}
|\xi(t)|\leq \alpha_1(L,\|B\|_{\infty},\|\dot
B\|_{\infty},\|u\|_2)\left(\int_0^t
|y(s)|^2\mr{d}s\right)^{1/2},
\ee
In view of \eqref{chap1l2xi}, $\|\delta x\|_2\leq \|\xi\|_2+\|B\|_{\infty}\|y\|_2.$ Conclude by \eqref{chap1l2xit} and this last inequality.
\end{proof}
}\fi

\begin{lemma}\label{chap1l1}
Let $\{y_k\}\subset L_2(a,b)$ be a sequence of continuous
non-decreasing functions that converges weakly to $y\in L_2(a,b).$
Then $y$ is non-decreasing.
\end{lemma}

\begin{proof}
Let $s,t\in (a,b)$ be such that $s<t,$ and $\varepsilon_0>0$
such that $s+\varepsilon_0<t-\varepsilon_0.$ For every $k\in \cN,$ and every
$0<\varepsilon<\varepsilon_0,$ the following inequality holds
\benl
\int_{s-\varepsilon}^{s+\varepsilon}
y_k(\nu)d\nu\leq \int_{t-\varepsilon}^{t+\varepsilon}
y_k(\nu)d\nu.
\eenl
Taking the limit as $k$ goes to infinity and
multiplying by $\frac{1}{2\varepsilon},$ we deduce that
\benl
\frac{1}{2\varepsilon}\int_{s-\varepsilon}^{s+\varepsilon}
y(\nu)d\nu\leq\frac{1}{2\varepsilon}\int_{t-\varepsilon}^{t+\varepsilon}
y(\nu)d\nu.
\eenl
As $(a,b)$ is a finite measure space, $y$ is a function of
$L_1(a,b)$ and almost all points in $(a,b)$
are Lebesgue points (see Rudin \cite[Theorem 7.7]{Rud87}). Thus, by taking $\varepsilon$ to 0, it follows from the previous inequality that
\benl
y(s)\leq y(t),
\eenl
which is what we wanted to prove.
\end{proof}

\begin{lemma}\label{chap1unifconv}
Consider a sequence $\{y_k\}$ of non-decreasing continuous functions in a
compact real interval $I$ and assume that $\{y_k\}$ converges
weakly to 0 in $L_2(I).$ Then it converges uniformly to 0 on any
interval $(a,b)\subset I.$
\end{lemma}

\begin{proof}
Take an arbitrary interval $(a,b)\subset I.$ First
prove the pointwise convergence of $\{y_k\}$ to 0. On
the contrary,  suppose that there exists
$c\in (a,b)$ such that $\{y_k(c)\}$ does not converge to 0. Thus
there exist $\varepsilon>0$ and a subsequence $\{y_{k_j}\}$
such that $y_{k_j}(c)>\varepsilon$ for each $j\in\cN,$ or
$y_{k_j}(c)<-\varepsilon$ for each $j\in\cN.$ Suppose, without loss of generality, that the first statement is true. Thus
\be
\label{chap1inequnifconv}
0< \varepsilon
(b-c)<y_{k_j}(c)(b-c)\leq\int_c^{b}y_{k_j}(t)\dtt,
\ee
where the last inequality holds since $y_{k_j}$ is
nondecreasing. But the right-hand side of \eqref{chap1inequnifconv} goes to 0 as $j$ goes to infinity. This contradicts the hypothesis and thus the pointwise
convergence of $\{y_k\}$ to 0 follows.
The uniform convergence is a direct consequence of the
monotonicity of the functions $y_k.$
\end{proof}

\begin{lemma}\cite[Theorem 22, Page 154 - Volume I]{DunSch58}
\label{chap1l3}
Let $a$ and $b$ be two functions of bounded variation in
$[0,T].$ Suppose that one is continuous and the other is
right-continuous. Then
\benl
\int_0^Ta(t)db(t)+\int_0^Tb(t)da(t)=\ds [ab]_{0-}^{T+}.
\eenl
\end{lemma}

\begin{lemma}\label{chap1Rpos}
Let $m=1, $ i.e.  consider a scalar control variable. Then, for any $\lambda\in \Lambda,$ the function $R\lam(t)$ defined in \eqref{chap1R} is continuous in $t.$
\end{lemma}

\begin{proof}
Consider definition \eqref{chap1SV}. Condition $V\lam\equiv 0$ yields $S\lam=C\lam B,$ and since $R\lam$ is scalar, we can write
\benl
R\lam=B^\top Q\lam B-2C\lam B_1-\dot C\lam B-C\lam\dot B.
\eenl
Note that $
B=f_1,$ $B_1=[f_0,f_1],$ $C\lam=-\psi f_1',$ and $Q\lam=-\psi(f_0''+\uh f_1''.$
Thus
\begin{align*}
R\lam
=&\psi(f_0''+\uh f_1'')(f_1,f_1)-2\psi
f_1'(f_0'f_1-f_1'f_0)\\
&+\psi(f_0'+\uh f_1')f_1'f_1-\psi f_1''(f_0+\uh
f_1)f_1-\psi f_1'f_1'(f_0+\uh f_1)\\
=&\psi[f_1,[f_1,f_0]].
\end{align*}
Since $f_0$ and $f_1$ are twice continuously differentiable, we conclude that $R\lam$ is continuous in time.
\end{proof}

\begin{lemma}\label{chap1legendre0}\cite{Hes51}
Consider a quadratic form $\Q=\Q_1+\Q_2$ where $\Q_1$ is a
Legendre form and $\Q_2$ is weakly continuous over some Hilbert space. Then $\Q$ is
a Legendre form.
\end{lemma}

\begin{lemma}\label{chap1legendre}\cite[Theorem 3.2]{Hes51}
Consider a real interval $I$ and a quadratic form $\Q$ over the Hilbert space
$L_2(I),$ given by
\benl
\Q(y):=\int_I y^\top(t)
R(t)y(t)
\dtt.
\eenl
Then $\Q$ is weakly l.s.c. over $L_2(I)$ iff
\be\label{chap1Rsucceq0}
R(t)\succeq 0,\quad \mr{a.e.}\ \mr{on}\ I.
\ee
\end{lemma}

\if{
\begin{proof}
Define
\be
\Q_1(y):=\half \int_0^T y^\top(t)R(t)y(t)\dtt,\quad
\mr{and}\quad
\Q_2(y):=\Q(y)-\Q_1(y).
\ee
Notice that $\Q_2$ is weakly continuous.
\noindent{\bf Part a.} Let us prove that if $\Q$ is a
Legendre form then \eqref{chap1Rsucceq0} holds. If $R(t)$ is not
positive definite a.e. on $[0,T]$ then there exist
$\beta>0$ and an interval $I\subset [0,T]$ of positive
measure such that
\benl
h^\top R(t)h\leq -\beta \|h\|^2,\quad \forall h\in \cR^m,\
\mr{a.e.}\ \mr{on}\ I.
\eenl
Define $\Uspace_2^I$ the subset of $\Uspace_2$ containing the
functions with support in $[0,T]\backslash I.$ As $\Uspace_2^I$
is an infinite dimension space, there exists an orthonormal
sequence $(y_k)\subset \Uspace_2^I.$ Easily follows that
$y_k\rightharpoonup 0$ in $\Uspace_2.$ As $\Q$ is w.l.s.c. we
get
\benl
0=\Q(0)\leq \liminf \Q(y_k)=\liminf \Q_1(y_k).
\eenl
Thus
\benl
0\leq \liminf \int_I y_k^\top(t)R(t)y_k(t)\dtt\leq
\liminf-\beta \int_I|y_k(t)|^2\dtt=-\beta,
\eenl
that leads us to a contraction. We conclude that condition
\eqref{chap1Rsucceq0} holds.
\noindent{\bf Part b.} Suppose now that condition
\eqref{chap1Rsucceq0} holds. Then $\sqrt{\Q_1}$ defines a norm
on $\Uspace_2$ equivalent to $\|.\|_2,$ and thus it is a
Legendre form. As $\Q_2$ is weakly continuous, we obtain by
Lemma \ref{chap1legendre0} thus $\Q$ is a Legendre form.
\end{proof}
}\fi

\begin{lemma}\label{chap1quadform}\cite[Theorem 5]{Dmi84}
 Given a Hilbert space $H,$ and $a_1,a_2,\hdots,a_p\in H,$ set
\benl
K:=\{x\in H:(a_i,x)\leq 0,\ \mr{for}\ i=1,\hdots,p\}.
\eenl
Let $M$ be a convex and compact subset of $\cR^s,$ and let
$\{Q^{\psi}:\psi\in M\}$ be a family of continuous
quadratic forms over $H$ with  the mapping $\psi \rightarrow
Q^{\psi}$ being affine. Set $M_{\#}:=\{ \psi \in M:\ Q^{\psi}\ \mr{is}\ \mr{weakly}\
\mr{l.s.c.}\}$ and assume that
\benl
\max_{\psi\in M} Q^{\psi}(x)\geq 0,\ \mr{for}\
\mr{all}\ x\in K.
\eenl
Then
\benl
\max_{\psi\in
M_{\#}} Q^{\psi}(x)\geq 0,\ \mr{for}\ \mr{all}\ x\in
K.
\eenl
\end{lemma}

\if{
\textbf{Part a.} Let us consider a sequence $(x_k)\subset
H$ such that $x_k\rightharpoonup 0,$ and
$(Q^{\psi}(x_{k}))$ converges for every $\psi\in M. $ We
can see that every sequence $(x_k)$ that converges strongly
to 0 satisfies this hypothesis. We associate to the
sequence a mapping $r:M\rightarrow \cR$ such that
$r(\psi)=\lim Q^{\psi}(x_k).$ We define the set $\R\subset
\cR$ as the set of these limits points.

Let us prove that $\R$ is a closed covex cone.

Closedness follows easily.

$\R$ is a cone because if a sequence $(x_k)$ converges
weakly to 0 and the $(Q^{\psi}(x_k))$ exists then any
positive multiple of $(x_k)$ satisfy the same properties.

Let us prove that it is convex. It remains only to prove
that the sum of 2 element of $\R$ is in $\R.$ Let us
consider $r,\rho\in \R$ such that to $r$ the sequence
$(x_k)$ is associated and $\rho$ we associate $(y_k).$ Let
$\psi\in M$ and consider $a^{\psi}$ the bilinear form
associated to $Q^{\psi}.$ There exists a sequence $(j_k)$
such that
\be
|a^{\psi}(x_k,y_{j_k})|\leq 2^{-k}.
\ee
Let us define $z_k:=x_k+y_{j_k}.$ Then $z_k\rightharpoonup
0,$ and
\be
Q^{\psi}(z_k)=Q^{\psi}(x_k)+2a^{\psi}(x_k,y_{j_k})+Q^{\psi}(y_{j_k})\rightarrow
r(\psi)+\rho(\psi).
\ee

We have then that there exists $(z_k)$ such that
$z_k\rightharpoonup 0,$ and
\benl
\lim Q^{\psi}(z_k)=r+\rho.
\eenl
We can conclude that $\R$ is convex.

\textbf{Part b.} We want to prove now that
\benl
\inf_{r\in \R}\max_{\psi \in M}(Q^{\psi}(x)+r(\psi))\geq
0,\quad x\in K.
\eenl
In order to achieve this take $x\in K$ and $r\in \R$ with
associated sequence $(x_k).$ Define $y_k:=x+x_k$ that
converges weakly to $x.$
\benl
Q^{\psi}(y_k)=Q^{\psi}(x)+2a^{\psi}(x,x_k)+Q(x_k)\rightarrow
Q^{\psi}(x)+r(\psi).
\eenl
By Hoffman Lemma, given $x\in K$ there exists a sequence
$(z_k)\subset K$ such that $y_k-z_k\rightarrow 0.$By
hypothesis
\benl
0\leq \max_{\psi\in M} Q^{\psi}(z_k)=\max_{\psi\in
M}\left[Q^{\psi}(y_k)+2a^{\psi}(y_k,z_k-y_k)+Q^{\psi}(z_k-y_k)
\right].
\eenl
As $M$ is compact we can take limit in the last expression
and obtain
\benl
0\leq \max_{\psi\in M}\left[Q^{\psi}(x)+r(\psi)\right].
\eenl
As the sequence $(x_k)$ considered is arbitrary we obtain
\benl
\inf_{r\in \R}\max_{\psi \in M}(Q^{\psi}(x)+r(\psi))\geq
0,\quad x\in K.
\eenl

\textbf{Part c.} We can apply the Minmax Theorem in \cite{Roc70} and get
\benl
\sup_{\psi\in M}\inf_{r\in
\R}\left[Q^{\psi}(x)+r(\psi)\right]\geq 0.
\eenl
Given $\psi\in M:$
\benl
\inf_{r\in \R}r(\psi)=
\left\{
\ba{cl}
0&\mr{si}\ r(\psi)\geq 0,\ \forall r,\\
\\
-\infty&\mr{si}\ \exists r/r(\psi)<0.
\ea
\right.
\eenl
Taking $x=0,$ $\sup_{\psi\in M}\inf_{r\in \R}r(\psi)\geq
0.$ Then it must exist $\psi\in M$ such that $r(\psi)\geq
0$ for every $r\in \R.$
We obtain for this $\psi:$
\benl
Q^{\psi}(0)= 0 \leq \inf_{r\in \R}r(\psi) = \liminf
Q^{\psi}(x_k).
\eenl
}\fi

The following result is an adaptation of Lemma 6.5 in \cite{Dmi87}.

\begin{lemma}\label{chap1etaPont}
Consider a sequence $\{v_k\}\subset \Uspace$ and $\{y_k\}$ their primitives defined by \eqref{chap1yxi}. Call $u_k:=\uh+v_k,$  $x_k$ its corresponding solution of
\eqref{chap1stateeq}, and let $z_k$ denote the linearized state
corresponding to $v_k,$ i.e. the solution of \eqref{chap1lineareq}.
Define, for each $k\in \cN,$
\be\label{chap1eta}
\delta x_k:=x_k-\xh,\quad \eta_k:=\delta x_k-z_k, \quad \gamma_k:=\gamma(y_k,y_k(T)).
\ee
Suppose that $\{v_k\}$ converges to 0 in the Pontryagin sense. Then
\begin{itemize}
\item[(i)]
\be\label{chap1etadot}
\dot\eta_k
=\sum_{i=0}^m\uh_if'_i(\xh)\eta_k+\sum_{i=1}^m
v_{i,k}f_i'(\xh)\delta x_k+\zeta_k,
\ee
\be\label{chap1deltaxdot}
\dot\delta x_k=\sum_{i=0}^m u_{i,k}f_i'(\xh)\delta
x_k+\sum_{i=1}^m v_{i,k}f_i(\xh)
+\zeta_k,
\ee
where $\|\zeta_k\|_{2}\leq o(\sqrt{\gamma_k})$ and $\|\zeta_k\|_{\infty}\rightarrow 0,$
\item[(ii)] $\|\eta_k\|_{\infty}\leq o(\sqrt{\gamma_k}).$
\end{itemize}
\end{lemma}

\begin{proof}
\textbf{(i,ii)}
Consider the second order  Taylor expansions of
$f_i,$
\benl
f_i(x_k)=f_i(\xh)+f'_i(\xh)\delta x_k+\half
f_i''(\xh)(\delta x_k,\delta x_k)+o(|\delta x_k(t)|^2).
\eenl
We can write
\be \label{chap1deltaxk1}
\dot\delta x_k=\sum_{i=0}^m u_{i,k}f_i'(\xh)\delta
x_k+\sum_{i=1}^m v_{i,k}f_i(\xh)
+\zeta_k,
\ee
with
\be\label{chap1zeta}
\zeta_k:=\half \sum_{i=0}^m u_{i,k}f_i''(\xh)(\delta
x_k,\delta x_k)+o(|\delta x_k(t)|^2)\sum_{i=0}^m u_{i,k}.
\ee
As $\{u_k\}$ is bounded in $L_{\infty}$ and $\|\delta x_k\|_{\infty}\rightarrow 0,$ we get $\|\zeta_k\|_{\infty}\rightarrow 0$ and the
following $L_2-$norm bound:
\be \label{chap1zetak1}
\begin{split}
\|\zeta_k\|_{2}
&\leq const.\sum_{i=0}^m\|u_{i,k}(\delta x_k,\delta
x_k)\|_2+o(\gamma_k)\|\sum_{i=0}^m u_{i,k}\|_1\\
&\leq const.\|u_k\|_{\infty}\|\delta
x_k\|_2^2= O({\gamma_k})\leq o(\sqrt{\gamma_k}).
\end{split}
\ee
Let us look for the differential equation of $\eta_k$
defined in \eqref{chap1eta}. By \eqref{chap1deltaxk1}, and adding and
substracting the term $\sum_{i=1}^m \uh_if_i'(\xh)\delta
x_k$ we obtain
\benl
\dot\eta_k
=\sum_{i=0}^m\uh_if'_i(\xh)\eta_k+\sum_{i=1}^m
v_{i,k}f_i'(\xh)\delta x_k+\zeta_k.
\eenl
Thus we obtain \textbf{(i).} Applying Gronwall's Lemma to this last differential
equation we get
\be\label{chap1etak1}
\|\eta_k\|_{\infty}\leq \|\sum_{i=1}^m
v_{i,k}f'_i(\xh)\delta x_k+\zeta_k\|_1.
\ee
Since $\|v_k\|_{\infty}<N$ and $\|v_k\|_1\rightarrow 0,$ we
also find that $\|v_k\|_2\rightarrow 0.$ Applying the Cauchy-Schwartz
inequality to \eqref{chap1etak1}, from \eqref{chap1zetak1} we get \textbf{(ii).}
\if{
From items \textbf{(i)} and \textbf{(ii)} we ontain easily  item \textbf{(iii).}
\noindent\textbf{(iv)} Observe that
\be\label{chap1Huxx1}
\intT (H_{uxx}\lam\delta x_k,\delta x_k,v_k)\dtt
=a_k+2b_k+c_k,
\ee
where
\be
a_k:=\intT (H_{uxx}\lam z_k,z_k,v_k)\dtt,\ b_k:=\intT (H_{uxx}\lam z_k,\eta_k,v_k)\dtt,
\ee
\be
c_k:=\int (H_{uxx}\lam\eta_k,\eta_k,v_k)\dtt.
\ee
Let us prove that $a_k\leq o(\gamma_k).$ Write
\be\label{chap1ak}
a_k=\intT  [(H_{uxx}\lam\xi_k,\xi_k,v_k)+2(H_{uxx}\lam\xi_k,By_k,v_k)+(H_{uxx}\lam By_k,By_k,v_k)]\dtt.
\ee
For the first term in \eqref{chap1ak} obtain
\begin{align}
\intT& (H_{uxx}
\lam\xi_k,\xi_k,v_k)\dtt
= [(H_{uxx}\lam\xi_k,\xi_k,y_k)]_0^T \\
&-\intT (\dot{H}_{uxx}\lam \xi_k,\xi_k,y_k)\dtt-2\intT  (H_{uxx}\lam(A\xi_k+B_1y_k),\xi_k,y_k)\dtt<o(\gamma_k).
\label{chap1ak1}\end{align}
For the second term in \eqref{chap1ak} we get
\begin{align}
\intT  &(H_{uxx}\lam\xi_k,By_k,v_k)\dtt
=\intT \xi^{\top}_kH_{uxx}\lam B\ddt\frac{y_k^2}{2}\dtt\\
&=[\xi^{\top}_kH_{uxx}\lam B \frac{y_k^2}{2}]_0^T
-\intT \ddt(\xi^{\top}_kH_{uxx}\lam B)\frac{y_k^2}{2}\dtt< o(\gamma_k),
\label{chap1ak2}\end{align}
where for the integrand we get
\benl
\ddt(\xi^{\top}_kH_{uxx}\lam B)=(\ddt \xi_k)^{\top}H_{uxx}\lam B+\xi^{\top}_k\ddt(H_{uxx}\lam B)\rightarrow 0\ \mr{in}\ L_{\infty},
\eenl
as $\|\xi_k\|_{\infty}\rightarrow 0,$ $\|y_k\|_{\infty}\rightarrow 0,$ and $H_{uxx}\lam B$ is a bounded variation function. For the third term in \eqref{chap1ak} we get
\begin{align}
\intT & (H_{uxx}\lam By_k,By_k,v_k)\dtt
=\intT B^{\top}H_{uxx}\lam B \ddt \frac{y_k^3}{3}\dtt\\
&=[B^{\top}H_{uxx}\lam B \frac{y_k^3}{3}]_0^T
-\intT \ddt(B^{\top}H_{uxx}\lam B)\frac{y_k^3}{3}\dtt <o(\gamma_k).
\label{chap1ak3}\end{align}
From \eqref{chap1ak1},\eqref{chap1ak2} and \eqref{chap1ak3} follows $a_k\leq (\gamma_k).$
Consider $b_k.$  Write
\be\label{chap1bk1}
b_k=\intT [(H_{uxx}\lam \xi_k,\eta_k,v_k)+(H_{uxx}\lam By_k,\eta_k,v_k)]\dtt.
\ee
Obtain for the first term in the integrand:
\begin{align}
\intT &(H_{uxx}\lam \xi_k,\eta_k,v_k)\dtt
=[(H_{uxx}\lam \xi_k,\eta_k,v_k)]_0^T
-\intT (\dot{H}_{uxx}\lam\xi_k,\eta_k,v_k)\dtt\\
&-\intT (H_{uxx}\lam(A\xi_k+B_1 y_k),\eta_k,y_k)\dtt
-\intT (H_{uxx}\lam \xi_k,\dot{\eta}_k,y_k)\dtt<o(\gamma_k),
\end{align}
where we used item \textbf{(iii)} for the last inequality. Go back to \eqref{chap1bk1} and get
\begin{align}
\intT (H_{uxx}\lam By_k,&\eta_k,v_k)\dtt
=\intT B^{\top}H_{uxx}\lam \eta_k\ddt \frac{y^2_k}{2}\dtt\\
&=[B^{\top}H_{uxx}\lam \eta_k \frac{y_k^2}{2}]_0^T
-\intT \ddt(B^{\top}H_{uxx}\eta_k)\frac{y_k^2}{2}\dtt<o(\gamma_k),
\end{align}
as $\ddt(B^{\top}H_{uxx}\eta_k)\rightarrow 0$ in $L_{\infty}$ in light of item \textbf{(iii)}. Thus $b_k\leq o(\gamma_k).$ For $c_k$ we get
\begin{align}
c_k=&\intT (H_{uxx}\lam \eta_k,\eta_k,v_k)\dtt
=[ (H_{uxx}\lam\eta_k,\eta_k,y_k)]_0^T\\
&-\intT (\dot{H}_{uxx}\lam \eta_k,\eta_k,y_k)\dtt
-2\intT (H_{uxx}\lam\dot{\eta}_k,\eta_k,y_k)\dtt
\leq o(\gamma_k).
\end{align}
Then $c_k\leq o(\gamma_k)$ as well and from \eqref{chap1Huxx1} we get item \textbf{(iv)}.
}\fi

\end{proof}

\section*{Acknowledgments}
The authors thank Professor Helmut Maurer for his useful remarks.


\medskip
Received July 2011; 1$^{st}$ revision February 2012; 2$^{nd}$ revision June 2012.
\medskip

\end{document}